\documentclass[12pt,reqno]{amsart}
\title{Boundary Representations of Locally Compact Hyperbolic Groups}
\date{September 22, 2023}
\usepackage{amsmath, amsthm, amsopn, amssymb, microtype}
\usepackage{mathrsfs} 
\usepackage[all,cmtip]{xy}
\usepackage{enumerate, tikz, etoolbox, intcalc, geometry, caption, subcaption, afterpage}
\usepackage[english]{babel}
\usepackage{bbm}
\usepackage{pdfpages}
\usepackage{tikz-cd}

\author{Michael Glasner}
\email{michael.glasner@weizmann.ac.il}
\address{Faculty of Mathematics and Computer Science, The Weizmann Institute of Science, 234 Herzl Street, Rehovot 7610001, ISRAEL}

\usepackage[colorlinks=true,urlcolor=blue,pdfborder={0 0 0}]{hyperref}
\hypersetup{linkcolor=[rgb]{0,0,0.6}}
\hypersetup{citecolor=[rgb]{0,0.6,0}}
\usepackage[nameinlink]{cleveref}
\crefname{theorem}{Theorem}{Theorems}
\crefname{theorem}{Theorem}{Theorems}
\crefname{mainthm}{Theorem}{Theorems}
\crefname{lemma}{Lemma}{Lemmas}
\crefname{lem}{Lemma}{Lemmas}
\crefname{remark}{Remark}{Remarks}
\crefname{prop}{Proposition}{Propositions}
\crefname{defn}{Definition}{Definitions}
\crefname{corollary}{Corollary}{Corollaries}
\crefname{cor}{Corollary}{Corollaries}
\crefname{section}{Section}{Sections}
\crefname{figure}{Figure}{Figures}
\crefname{quest}{Question}{Questions}

\newcommand{\N}{\mathbb{N}}
\newcommand{\Z}{\mathbb{Z}}

\newcommand{\R}{\mathbb{R}}

\newtheorem{theorem}{Theorem}[section]
\newtheorem{lemma}[theorem]{Lemma}

\newtheorem{corollary}[theorem]{Corollary}

\theoremstyle{definition}
\newtheorem{definition}[theorem]{Definition}
\newtheorem{example}[theorem]{Example}

\newtheorem*{conjecture}{Conjecture}

\newcommand{\acts}{\operatorname{\curvearrowright}}

\newcommand{\abs}[1]{\left| #1 \right|}

\begin{document}

\maketitle

\begin{abstract}
    We develop the theory of Patterson-Sullivan measures on the boundary of a locally compact hyperbolic group, associating to certain left invariant metrics on the group measures on the boundary. We later prove that for second countable, non-elementary, unimodular locally compact hyperbolic groups the associated Koopman representations are irreducible and their isomorphism type classifies the metric on the group up to homothety and bounded additive changes,  generalizing a theorem of Garncarek on discrete hyperbolic groups. We use this to answer a question of Caprace, Kalantar and Monod on type I hyperbolic groups in the unimodular case.
\end{abstract}

\tableofcontents

\section{Introduction}
\subsection{Koopman and boundary representations}

A locally compact group $G$ is called hyperbolic if for some (hence any) compact generating set the coresponding word metric is Gromov hyperbolic. Equivalently $G$ is hyperbolic if it admits a proper, cocompact, isometric action on a proper, geodesic, Gromov hyperbolic metric space $X$. 
Examples are:

\begin{enumerate}
    \item Rank one simple Lie groups with finite center.
    \item Minimal parabolic subgroups of rank one simple Lie groups with finite center.
    \item Groups acting properly and co-compactly on locally finite trees.
\end{enumerate}

One of the most useful tools to study $G$ is the action of $G$ on its Gromov boundary $\partial G$. Given a left invariant metric $d$ on $G$ which is quasi isometric to a word metric and satisfies some mild assumptions we will construct a measure $\mu$ on $\partial G$ called the Patterson-Sullivan measure associated to $d$. The measure class $[\mu]$ is invariant and $\mu$ is quasi conformal, i.e there exists $C \geq 1$ such that for all $g \in G$:
$$ \frac{1}{C}e^{-h(|g|-2(g,\xi))} \leq \frac{dg_*\mu}{d\mu}(\xi) \leq C e^{-h(|g|-2(g,\xi))}$$
where $h$ is the critical exponent of $G$ with respect to $d$ and $(\cdot,\cdot)$ is the Gromov product.

To such a nonsingular action one can associate a unitary representation of $G$ on $L^2(\mu)$ called the Koopman representation, defined by:
$$[\pi(g)f](\xi) = \sqrt{\frac{dg_*\mu}{d\mu}(\xi)}f(g^{-1}\xi)$$
The Koopman representation depends up to unitary equivalence only on the measure class $[\mu]$.

Irreducibility of the Koopman representation implies ergodicity and can be thought of as a mixing property of the action.

We will call the Koopman representations associated to the Patterson-Sullivan measures boundary representations or Patterson-Sullivan representations.

Our first main theorem is the following:
\begin{theorem} \label{irreducibility main theorem}
Let $G$ be a second countable, unimodular, non-elementary locally compact hyperbolic group.
For any Borel measurable, left invariant metric $d$ on $G$ which is Gromov hyperbolic and quasi isometric to a word metric, the associated Patterson-Sullivan representation $\pi$ is irreducible.
\end{theorem}

Our second main theorem provides a classification of the Patterson-Sullivan representations in terms of the metric $d$.

\begin{theorem} \label{classification main theorem}
    Let $G$ be a second countable, unimodular, non-elementary locally compact hyperbolic group. Given two Borel measurable, left invariant metrics $d_1,d_2$ on $G$ which are Gromov hyperbolic and quasi isometric to a word metric, the corresponding Patterson-Sullivan representations $\pi_1$ and $\pi_2$ are unitarily equivalent if and only if $d_1$ and $d_2$ are roughly similar, that is, there exist constants $L,C>0$ such that for all $g,h \in G$, $Ld_1(g,h) - C \leq d_2(g,h) \leq Ld_1(g,h) + C$. (Note this is much stronger than a quasi-isometry since the multiplicative constant is the same in the upper and lower bounds!)

\end{theorem}

These theorems generalize earlier works by Bader, Muchnik \cite{BAM11}  and Garncarek \cite{G14} which deal with the cases of fundamental groups of negatively curved manifolds and discrete hyperbolic groups respectively. Many of our methods are similar to those in \cite{G14} although some constructions are altered to fit the locally compact setting. It is interesting to notice throughout the paper the key lemmas in which we use the unimodularity assumption from theorems \ref{irreducibility main theorem} and \ref{classification main theorem}. Specifically note lemma \ref{cancellation lemma}, lemma \ref{uniformly bounded} and appendix \ref{appendix b}.

Theorem \ref{irreducibility main theorem} can be envisioned in the wider context of a general conjecture by Bader and Muchnik:

\begin{conjecture}[Bader, Muchnik \cite{BAM11}]
Let $G$ be a locally compact group and $\nu$ a spread out probability measure on $G$. For any $\nu$-boundary of $G$ the associated Koopman representation is irreducible.
\end{conjecture}

\subsection{Type I hyperbolic groups}

Our main application of theorems \ref{irreducibility main theorem} and \ref{classification main theorem} will be to the study of type I hyperbolic groups, generalizing results of Caprace, Kalantar and Monod from \cite{CKM21}.

\begin{definition}
A locally compact group $G$ is of type I if any two irreducible unitary representations of $G$ which are weakly equivalent (see \cite[Appendix~F.1]{BDV08}) are unitarily equivalent.
\end{definition}

Type I groups have a relatively simple unitary representation theory. In a certain sense they are the groups for which there is hope of obtaining a complete classification of all unitary representations. For an introduction on type I groups see \cite[Section~7.2]{FOL16}.

In section \ref{type I hyperbolic groups} we use theorems \ref{irreducibility main theorem} and \ref{classification main theorem} to deduce the following generalization of \cite[Theorem~B]{CKM21}:

\begin{theorem} \label{type I theorem}
    Let $G$ be a second countable unimodular locally compact hyperbolic group. If $G$ is of type I then $G$ has a co-compact amenable subgroup.
\end{theorem}

By a theorem of Thoma \cite{Th64} \cite{Th68}, discrete groups are of type I if and only if they are virtually abelian. Therefore discrete hyperbolic groups of type I are elementary. As a result the above theorem is interesting only for non discrete groups.

In \cite[Theorem~D]{CCMT15} there is a classification of non amenable hyperbolic locally compact groups containing a cocompact amenable subgroup. We can use this to deduce:
\begin{corollary}
    Let $G$ be a second countable unimodular locally compact hyperbolic group. Recall $G$ contains a unique maximal compact normal subgroup $W$. If $G$ is of type I then exactly one of the following holds:

    \begin{enumerate}
        \item $G/W$ is a rank one adjoint Lie group.
        \item $G/W$ is a closed subgroup of the automorphism group of a locally finite, non elementary tree $T$, acting without inversions, with exactly two orbits of vertices and $2$-transitively on the boundary.
        \item $G/W$ is trivial or isomorphic to $\Z,\R,\Z \rtimes \{ \pm 1\}$ or $\R \rtimes \{ \pm 1\}$.
    \end{enumerate}

\end{corollary}
It is already pointed out in \cite[Remark~5.6]{CKM21} that generalizing the works of Garncarek \cite{G14} as in theorems \ref{irreducibility main theorem} and \ref{classification main theorem} would imply the theorems above.

Theorem \ref{type I theorem} fits in to a much wider structural conjecture of Caprace Kalantar and Monod about type I groups:

\begin{conjecture} [Caprace, Kalantar, Monod \cite{CKM21}]
    Every second countable locally compact group of type I admits a cocompact amenable subgroup.
\end{conjecture}

\subsection{The space $\partial^2 G$}
Let $G$ be a second countable locally compact hyperbolic group, $d$ a Borel measurable left invariant metric on $G$ which is quasi isometric to a word metric and Gromov hyperbolic. Denote by $\mu$ the corresponding Patterson Sullivan measure.

During the proof of theorem \ref{classification main theorem} we will establish several results of independent interest about the space $\partial^2 G$ of distinct pairs in $\partial G$.

First we shall construct an invariant (infinite) measure $m$ in the measure class of $\mu^2$. This measure will be the analogue of the Bowen-Margulis-Sullivan measure in our case. 

It will then be shown that if $G$ is unimodular, then the action of $G$ on $(\partial^2 G , m)$ is weakly mixing in the following sense:

\begin{theorem} \label{weak mixing introduction theorem}
    For any ergodic p.m.p action $G \acts (\Omega,\omega)$ the diagonal action of $G$ on $(\partial^2 G \times \Omega,m\times \omega)$ is ergodic.
\end{theorem}

This implies in particular that the action of $G$ on $(\partial^2 G,m)$ is ergodic and therefore that $m$ is the unique invariant measure in the measure class of $[\mu^2]$. This will be the critical fact about $\partial^2 G$ used in the proof of theorem \ref{classification main theorem}.

In order to prove theorem \ref{weak mixing introduction theorem} we define a cocycle $\tau:G\times\partial^2 G \to \R$:

$$\tau(g,\xi,\eta) = \frac{1}{h}\left( ln\frac{dg^{-1} _*\mu}{d\mu}(\eta)  - ln\frac{dg^{-1} _*\mu}{d\mu}(\xi)\right)$$
where $h$ is the critical exponent of $(G,d)$.

$G$ then acts on $\partial^2 G \times \R \times \Omega$ via:

$$g(\xi,\eta,t,w) = (g\xi,g\eta,t+\tau(g,\xi,\eta),gw)$$
and this action commutes with the $\R$-flow defined by:

$$\Phi^s (\xi,\eta,t,w) = (\xi,\eta,t + s,w)$$

This flow provides an analogue of the geodesic flow of a negatively curved simply connected manifold in our general context. Since it commutes with the $G$ action $\Phi$ descends to an $\R$-flow $\phi$ on the space $X = (\partial^2 G \times \R \times \Omega)//G$ of $G$-ergodic components. We will show that $X$ supports a $\phi$-invariant probability measure $\nu$ in its canonical measure class.

Using $\phi$ we will prove the following ergodic theorem for $\partial^2 G$:

\begin{theorem}
    Suppose $G$ is unimodular and let $G \acts (\Omega,\omega)$ be an ergodic p.m.p action.
    For any $f \in L^1(\partial^2 G \times \Omega,m \times \omega)$, for almost every $(\xi,\eta,w)$ and any $a \in \R$:
    $$\lim_{b \to \infty} \frac{1}{b-a}\int_{\{g \in G| \tau(g,\xi,\eta) \in [a,b]\}} f(g\xi,g\eta,gw) d\lambda(g) = \int f(\xi,\eta,w) dm\times\omega$$
    
\end{theorem}

We will then use this to prove theorem \ref{weak mixing introduction theorem}.
Our methods will be closely based on a paper by Bader and Furman \cite{BF17}, generalizing some of their results from discrete to unimodular locally compact hyperbolic groups. While the proofs are similar, the non-discreteness of the group 
makes the measure theoretic aspects much more delicate.

\subsection{Non unimodular groups}
It is interesting to consider what happens in theorems \ref{irreducibility main theorem} and \ref{classification main theorem} when the unimodularity assumption is dropped. The following example shows that in general theorem \ref{invariant measure} dose not hold:

\begin{example}
Consider the minimal parabolic subgroup $P$ of $SL_2(\R)$, i.e the group of $2$ by $2$ upper triangular matrices with determinant $1$. $P$ is non-unimodular and non-elementary. This group is isomorphic to the group of affine transformations of the real line with positive leading coefficient. The Gromov boundary of this group is the boundary of the hyperbolic plane. In the upper half plane model $P$ fixes $\infty$ and acts affinly on $\R$. We shall see in \ref{ahlfors regular} that any Patterson-Sullivan measure corresponding to any metric on $P$ has no atoms and therefore is supported entirely on $\R$. Since the action of $P$ on $\R$ is transitive there is a unique $P$-invariant measure class on $\R$, the Lebesgue class. Thus the corresponding representation is independent of the metric $d$ and equivalent to the Koopman representation of $P$ on $L^2(\R)$. Because the leading coefficient of any affine map in $P$ is positive the action on $L^2(\R)$ preserves the subspaces of functions whose Fourier transforms are supported on $\R_+$ and $\R_-$ so this representation is not irreducible. For a full description of the representation theory of $P$ see \cite[Section~6.7]{FOL16}.
\end{example}

Amenable hyperbolic locally compact groups are unimodular if and only if they are elementary(\cite[Theorem~7.3]{CCMT15}). Since $P$ is amenable it is very natural to ask whether the unimodularity assumption in theorems \ref{irreducibility main theorem} and \ref{classification main theorem} can be replaced by the assumption that the group is non-amenable. We do not know the answer to this question. Next we give an example of a non-amenable, non-unimodular, hyperbolic locally compact group.

\begin{example}
Let $T$ be a regular tree of degree $n<\infty$. Choose an orientation $O$ on the edges of $T$ such that around each vertex, $k$ edges are oriented inward and $n-k$ outward. consider the group of automorphisms of the tree preserving the orientation, $G = Aut(T,O)$. If $k \neq n-k$ then $G$ is non-unimodular, since the stabilizer of an edge has a different index in the stabilizers of the corresponding vertices, even though the stabilizers of the two vertices are conjugate. As long as $k,n-k \neq 1$, $G$ is non-amenable since it is non-elementary and does not fix any end of the tree.
As an example of the non-unimodular non-amenable case it is interesting to ask whether $G$ satisfies theorems \ref{irreducibility main theorem} and \ref{classification main theorem}. As far as we are aware it is not even known whether the Koopman representation of $G$ corresponding to the standard measure on $\partial T$ is irreducible.
\end{example}

\subsection{Structure of the paper}
In section \ref{prelims} we fix the notation and our conventions for hyperbolic metric spaces and we give an overview of the theory of locally compact hyperbolic groups. In section \ref{PSmeasures} we develop the theory of Patterson-Sullivan measures for locally compact hyperbolic groups. The theory is very similar to the discrete case but has never been written down before. The main parts of the proof of theorem \ref{irreducibility main theorem} are given in section \ref{growth setimates} and section \ref{constructing operators}. Section \ref{growth setimates} Contains many of the geometric aspects of the proof while section \ref{constructing operators} includes the representation theoretic sides. In section \ref{irreducibility} we complete the proof of theorem \ref{irreducibility main theorem} and in section \ref{rough equivalence of metrics} we prove theorem \ref{classification main theorem} on rough equivalence of metrics. Section \ref{type I hyperbolic groups} contains applications to hyperbolic groups of type I. In section \ref{the geodesic flow} and section \ref{double ergodicity} we construct a measurable version of a geodesic flow for $G$ and use it to prove the ergodicity of the action on $\partial^2 G$, a result which is needed in section \ref{rough equivalence of metrics}. In the appendices \ref{appendix a} and \ref{appendix b} we deal with measure theoretic technicalities arising from the fact that our groups are non discrete.

\subsection{Acknowledgments}
I would like to thank my advisor, Uri Bader, for his constant support and guidance. Every conversation with him leaves me in a better mood and in awe of some amazing mathematics. Without him this project would never have come to fruition. I would also like to thank my father, Yair Glasner, who has taught me so much both in math and in general through the years and my mother, Shalvia Glasner, for listening to me and helping me when things seemed hard and unapproachable. Last but not least, I would like to thank my teammates at the Weizmann institute: Alon Dogon, Itamar Vigdorovitch, Sheve Leibtag, Aviv Taller, Benny Bachner, Paul Vollrath, Raz Slutsky, Gil Goffer, Tal Cohen, Guy Salomon, Omer Lavi, Guy Kapon, Yuval Salant, Yuval Gorfine, Idan Pazi and Peleg Bar-Sever for their constant friendship and support.

\section{Preliminaries} \label{prelims}
\subsection{Notation and conventions}
Throughout the paper we will use estimates involving additive and multiplicative constants. In order to stop these constants from snowballing we introduce the following notation. Given functions $f$ and $g$ with a common domain, if there exists $0\leq c$ such that $f \leq g + c$ we write $f \lesssim g$. If $f \lesssim g$ and $g \lesssim f$ we write $f \approx g$. Similarly, if there exists $0<C$ such that $f \leq Cg$ we write $f \prec g$ and if $f \prec g$  and $g \prec f$ we write $f \asymp g$. If the constants in an estimate depend on a parameter we write the parameter as subscript in the inequality to denote the fact that the estimate is not uniform in that parameter. For example $f(x,y) \prec_y g(x,y)$ means that there exists $0 < C = C(y)$, possibly depending on $y$ but not depending on $x$ such that $f(x,y) \leq C(y)g(x,y)$. The same conventions are used in \cite{G14}.

\subsection{Maps between metric spaces}
Given two metric spaces $(X,d_X)$ and $(Y,d_Y)$ we call a function $f:X\rightarrow Y$ an (L,c)-quasi-isometric embedding if $$\frac{1}{L}d_X(x_1,x_2) - c \leq d_Y(f(x_1),f(x_2) \leq Ld_X(x_1,x_2) + c$$ A (L,c)-quasi-isometric embedding is called a (L,c)-quasi-isometry if there exists $0\leq D$ such that for any $y\in Y$ there exists $x\in X$ such that $d_Y(f(x),y) \leq D$. A (1,c)-quasi-isometric embedding is called a c-rough embedding and a (1,c)-quasi-isometry is called a c-rough isometry. A map $f:X\rightarrow Y$ is a (L,c)-rough similarity if $Ld_X(x_1,x_2) - c \leq d_Y(f(x_1,x_2) \leq Ld_X(x_1,x_2) + c$. When the constants are irrelevant we will omit them and say "f is a quasi-isometry" or "f is a rough embedding".

A geodesic in a metric space X is an isometric embedding of $\mathbb{R}$. We define quasi-geodesics and rough geodesics as quasi-isometric embeddings and rough embeddings of $\mathbb{R}$. We make the same definitions for rays and segments replacing $\R$ by closed rays and closed intervals in $\R$. A metric space is called geodesic if every two points can be connected by a geodesic segment. Similarly a metric space is called (L,c)-quasi-geodesic if any two points can be connected by a (L,c)-quasi-geodesic segment and c-roughly geodesic if it is (1,c)-quasi-geodesic. Here also we will omit the constants if they are irrelevant. Given a quasi-geodesic or roughly geodesic space we will implicitly assume that all quasi-geodesics or rough geodesics used have uniformly bounded constants compatible with the constants of the space.

\subsection{Hyperbolic metric spaces}
A good introduction to the theory of hyperbolic metric spaces can be found in \cite[Chapter~3.H]{BH99}. A more general theory for non-proper non-geodesic spaces is developed in \cite{DSU17}. We provide here a brief survey.

Let (X,d) be a metric space. Given $x,y,z \in X$ define the \textbf{Gromov product} of $x$ and $y$ with respect to $z$: $$(x,y)_z = \frac{1}{2}(d(z,x) +d(z,y) - d(x,y))$$ Note that $|(x,y)_z - (x,y)_w| \leq d(w,z)$. If $X$ is equipped with a base point $o$ we will denote $(x,y) = (x,y)_o$ 

\begin{definition}
Let $0\leq \delta$. A metric space (X,d) is called $\delta$-hyperbolic if for any $x,y,z,o \in X$ $$(x,z)_o \geq \min\{(x,y)_o, (y,z)_o\} - \delta$$ A metric space is (Gromov)-hyperbolic if it is $\delta$-hyperbolic for some $0\leq \delta$.
\end{definition}

Given a hyperbolic space $X$ fix a base point $o$. A sequence $(x_n)$ is called \textbf{Cauchy-Gromov} if $\lim_{n,m \to \infty} (x_n,x_m) = \infty$. We say such a sequence converges to infinity. Two such sequences $(x_n),(y_n)$ are said to be equivalent if $\lim_{n \to \infty} (x_n,y_n) = \infty$. Using hyperbolicity of the space one sees that this is an equivalence relation. These notions don't depend on $o$.

\begin{definition}
The Gromov boundary of a hyperbolic space $X$ is 
$$\partial X = \{[(x_n)]| (x_n) \text {is Cauchy-Gromov}\}$$
We denote $\bar{X} = X \cup \partial X $.
\end{definition}

We now extend the definition of the Gromov product to points in $\bar{X}$. Given $\xi,\eta \in \partial X$ and $z,o \in X$ define:
$$(\xi,\eta)_o  = \inf \liminf_{n \to \infty} (x_n,y_n)_o $$ $$(\xi,z)_o  = \inf \liminf_{n \to \infty} (x_n,z)_o$$ 
where the infimum is taken over sequences representing the boundary points. After replacing $\delta$ by $2\delta$ the extension of the Gromov product still satisfies the condition given in the definition of hyperbolicity. It follows from hyperbolicity that

$$ |\sup \limsup_{n \to \infty} (x_n,y_n)_o - \inf \liminf_{n \to \infty} (x_n,y_n)_o| < 2\delta $$

$$|\sup \limsup_{n \to \infty} (x_n,z)_o - \inf \liminf_{n \to \infty} (x_n,z)_o| < 2\delta $$

Thus up to $2\delta$ one can compute the Gromov product by any choice of sequences representing the points. We will only care about the Gromov product up to an additive constant so this will be useful.

 $\bar{X}$ can be given a topology as follows. Choose a base point $o$ for $X$. Open sets in $X$ are just the open sets in the metric topology on $X$. A family of basic (not necessarily open) neighborhoods around a point $\xi \in \partial X$ are given by $\{x \in \bar{X} | (x,\xi) > M\}$ where $M \in [0,\infty)$. In this topology $x_n$ converges to $\xi \in \partial X$ if and only if $(x_n,\xi) \rightarrow \infty$. This topology does not depend on the base point $o$ and if $X$ is a proper geodesic metric space then $\partial X$ and $\bar{X}$ are compact.

\begin{lemma}
    The Gromov product is lower semi continuous as a function $(\cdot,\cdot)_\cdot : \bar{X} \times \bar{X} \times X \rightarrow \R \cup \{\infty\}$.
\end{lemma}

The proof can be found in \cite[Lemma~3.4.23]{DSU17}

\begin{definition}
Let $X$ be a hyperbolic metric space. A visual metric $d_\varepsilon$ with parameter $\varepsilon > 0$ on $\partial X$ is a metric satisfying: 

\begin{equation} \label{visual_metric}
   d_\varepsilon(\xi,\eta)\asymp e^{-\varepsilon(\xi,\eta)}
\end{equation}

\end{definition}

For small enough $\varepsilon$ a visual metric always exists (see \cite[Proposition~3.6.8]{DSU17}, \cite[Chapter~3.H.3]{BH99}). Such a metric always generates the standard topology on $\partial X$. The hyperbolicity implies that a visual metric $d_\varepsilon$ is almost an ultra-metric in the following sense, for all $\xi_1$, $\xi_2$, $\xi_3$:
$$d_\varepsilon(\xi_1,\xi_3) \prec max \{d_\varepsilon(\xi_1,\xi_2), d_\varepsilon(\xi_2,\xi_3)\}$$

\begin{definition}
Let X be a hyperbolic metric space, $x,o \in X$ and $\sigma > 0$. The shadow cast by $x$ from $o$ with parameter $\sigma$ on $\partial X$ is 
$$\Sigma_o(x,\sigma) = \{\xi \in \partial X | (x,\xi)_o + \sigma > d(x,o)\}$$
If o is a given base point of X we will omit it from the notation.
\end{definition}

We now mention three lemmas we will need in order to work with hyperbolic locally compact groups. These lemmas are standard for geodesic hyperbolic spaces but we will need them in a more general setting. The proofs of these lemmas are based on the ones given in \cite[Subsection~3.1]{G14} and are basically a restating of results from \cite{BHM11} and \cite{BS00}

\begin{lemma} \label{hyperbolicQI}
Let $X$ be a geodesic hyperbolic metric space and $Y$ a metric space which is quasi-isometric to $X$, then $Y$ is hyperbolic if and only if it is roughly geodesic.
\end{lemma}

\begin{proof}
We follow the proof in \cite[Subsection~3.1]{G14}. By \cite[Theorem~A.1]{BHM11} $Y$ is hyperbolic if and only if $Y$ is quasi-ruled in the sense defined in \cite{BHM11}. By \cite[Lemma~A.2]{BHM11} a quasi-ruled space is roughly geodesic, on the other hand a roughly geodesic space is obviously quasi-ruled.
\end{proof}

\begin{lemma}
Let X be a roughly geodesic hyperbolic space then:
\begin{enumerate}
    \item For any $x \in X$ and $\xi \in \partial X$ there is a roughly geodesic ray $\gamma$ with $\gamma(0) = x$, $\gamma(\infty) = \lim_{t \to \infty}\gamma(t) = \xi$.
    \item For any $\xi , \eta \in \partial X$ there is a rough geodesic $\gamma$ with $\gamma(-\infty) = \xi$, $\gamma(\infty) = \eta$.
\end{enumerate}
\end{lemma}

\begin{proof}
This follows from \cite[Proposition~5.2]{BS00} and from the obvious fact that roughly geodesic spaces are almost geodesic in the sense defined in \cite{BS00}. Note that as a consequence of \cite[Proposition~5.2]{BS00} the almost geodesic hyperbolic spaces are precisely the roughly geodesic ones.  
\end{proof}

\begin{lemma} \label{boundary invariance}
If $f: X \to Y$ is a quasi-isometry of roughly geodesic hyperbolic spaces, $o \in X$, then  there exists $L,C > 0$ such that $\frac{1}{L}(x,y) - C \leq (f(x),f(y))_{f(o)}\leq L(x,y)_o + C$. As a result $f$ canonically extends to a homeomorphism $\partial f: \partial X \to \partial Y$.
\end{lemma}

\begin{proof}
    The estimate on Gromov products is exactly \cite[Proposition~5.5]{BS00}. It follows that $f$ respects the notion of being Cauchy-Gromov so $f$ extends to the boundary, furthermore the estimate implies this extension is continuous. Applying the same argument to a quasi-inverse of $f$ shows the map between boundaries is a homeomorphism. 
\end{proof}

Finally we will need one last lemma.

\begin{lemma} \label{shadow contains ball}
Let $X$ be a roughly geodesic hyperbolic space with base point $o$, fix a visual metric $d_\epsilon$ on $\partial X$. Let $C\geq0$. For large enough $\sigma$ (depending on $C$), for any $\xi \in \partial X$, $R>0$, any roughly geodesic ray $\gamma$ with $\gamma(0) = o,  \gamma(\infty) = \xi$ and for any $g$ such that $d(\gamma(R),g) \leq C$:
$$B_{e^{-\epsilon R}}(\xi) \subseteq \Sigma_o(g,\sigma)$$

In particular for large enough $\sigma$: $$B_{e^{-\epsilon R}}(\xi) \subseteq \Sigma_o(\gamma(R),\sigma)$$
\end{lemma}

\begin{proof}
On the one hand $(\gamma(R),\xi)\approx |\gamma(R)| \approx R$. On the other hand if $d_\epsilon (\xi,\eta) < e^{-\epsilon R}$ then $(\xi,\eta) \gtrsim R$ and since $d(\gamma(R),g) < C$ , $(\gamma(R),g) \approx_C R$. Now we get $(g,\eta) \gtrsim min\{(g,\gamma(R)),(\gamma(R),\xi),(\xi,\eta)\} \approx_C R$. Thus if $\sigma$ is chosen large enough (independently of $o,\xi,R,\gamma$ but depending on 
$C$) we get $\eta \in \Sigma_o(\gamma(R),\xi)$. Therefore $B_{e^{-\epsilon R}}(\xi) \subseteq \Sigma_o(g,\sigma)$ as needed.
\end{proof}

\subsection{Locally compact hyperbolic groups}

This subsection will present the framework of locally compact hyperbolic groups in which we will work. A good introduction to the concept of a locally compact hyperbolic group can be found in \cite{CCMT15}. We will use the theory developed there.

\begin{definition}
    A locally compact group $G$ is called hyperbolic if it is compactly generated and for some (hence any) compact generating set the induced word metric on $G$ is hyperbolic.
\end{definition} 

In the discrete case this definition recovers the standard notion of hyperbolicity for groups. By \cite[Corollary~2.6]{CCMT15} $G$ is hyperbolic if and only if there exists a proper geodesic metric space $X$ on which $G$ admits a continuous proper co-compact isometric action. The action of G extends continuously to $\bar{X}$. We obtain the following examples:

\begin{enumerate}
    \item Rank one simple Lie groups with finite center are hyperbolic.
    \item Minimal parabolic subgroups of rank one simple Lie groups with finite center are hyperbolic.
    \item Groups acting properly and co-compactly on locally finite trees are hyperbolic.
    \item Generalizing the previous examples, locally compact groups acting isometrically properly and co-compactly on proper CAT(-1) spaces are hyperbolic.
    \item Totally disconnected compactly generated groups are hyperbolic if and only if some (and hence every) associated Cayley-Abels graph is hyperbolic.
\end{enumerate}

Let $G$ be a locally compact hyperbolic group. Denote by $\mathcal{D}(G)$ the set of all left invariant hyperbolic metrics on $G$ which are quasi-isometric to a word metric on G and are Borel measurable as functions $d: G \times G \rightarrow \R$. If $d \in \mathcal{D}(G)$ we will always consider $(G,d)$ to be equipped with base point $1 \in G$. Since word metrics are roughly isometric to the corresponding Cayley graph which is a geodesic space, it follows from \ref{hyperbolicQI} that all the metrics in $\mathcal{D}(G)$ are roughly geodesic. In addition by \ref{boundary invariance} we see that all such metrics give rise to the same boundary which we shall denote $\partial G$. Similarly if $G$ acts on $X$ as above, $\partial G \cong \partial X$. The following are examples of metrics in $\mathcal{D}(G)$:

\begin{enumerate}
    \item Word metrics corresponding to compact generating sets.
    \item If $G$ acts on $(X,d_X)$ as above, for any $x \in X$ one can define $d_0(g,h) = d_X(gx,hx)$. In general this is only a pseudo-metric but any left invariant Borel measurable metric roughly similar to it will be in $\mathcal{D}(G)$. There always exists such a metric, for example one can define $d(g,g) = 0, d(g,h) = d_0(g,h) + 1$ for $g \neq h$. Alternatively one can adjust the theory to work with pseudo-metrics.
\end{enumerate}

Note that we do not require the metrics in $\mathcal{D}(G)$ to be continuous, the reason for this is that we will only be interested in the coarse geometry of the group and won't care about local properties.

\begin{lemma}
    Let $G$ be a locally compact hyperbolic group, then $\partial G$ is compact and the action of $G$ is continuous. 
\end{lemma}

\begin{proof}
By \cite[Proposition~2.1]{CCMT15} there exists a space $X$ with a $G$ action as above. Since $X$ is proper and geodesic $\partial X$ is compact and the action on $\partial X$ is continuous, but $\partial G \cong \partial X$.
\end{proof}

The following lemma is very important and expresses algebraically the contracting dynamics of $G$ on $\partial G$:

\begin{lemma} \label{basic contraction lemma}
    Let $G$ be a locally compact hyperbolic group and $d \in \mathcal{D}(G)$ then for $g \in G$ and $\xi \in \overline{G}$:
    $$(g,g\xi) \approx |g| - (g^{-1},\xi)$$
    uniformly in $g$ and $\xi$.
\end{lemma}

\begin{proof}
    The proof is taken from \cite[Lemma~4.1]{G14}. If $h \in G$ then a direct calculation shows that $(g,gh) = |g| - (g^{-1},h)$. Letting $h$ converge to $\xi \in \bar{G}$ we get the desired result.
\end{proof}

We get the following corollary:

\begin{lemma}(\cite[Lemma~4.1]{G14}) \label{nonempty shadow} 
    Let $G$ be a locally compact hyperbolic group and $d \in \mathcal{D}(G)$ then for $g \in G$:
    $$|g| \approx \sup_{\xi \in \partial G} (g,\xi)$$
    uniformly in $g$.
\end{lemma}
 
\begin{proof}
The proof is taken from \cite[Lemma~4.1]{G14}. Obviously $(g,\xi) \leq |g|$ for all $g$. Let $\xi_1,\xi_2 \in \partial G$ be two distinct points. By \ref{basic contraction lemma} for any $g \in G$:
$$\max_i (g,g\xi_i) \approx |g| - \min_i (g^{-1},\xi_i)$$
But $(\xi_1,\xi_2) \gtrsim \min\{(g^{-1},\xi_1),(g^{-1},\xi_2)\}$, so for some $i$, $ (g,g\xi_i) \approx |g|$.

\end{proof}

Let $d \in \mathcal{D}(G)$. Even though the metric topology induced on $G$ by $d$ need not be proper or geodesic the space $\bar{G} = G \cup \partial G$ can be given a compact topology agreeing with the locally compact topology on $G$. Furthermore this topology does not depend on $d$. To do this one sets a base of (not necessarily open) neighborhoods of $\xi \in \partial G$ to be $\{x \in \bar{G}|(x,\xi) > M\}$. So the topology is the regular topology on $G$ and $x_n \to \xi$ if and only if $(x_n,\xi) \to \infty$.

\begin{lemma} \label{compactification of G}
The topology described above is independent of $d$ and compact.
\end{lemma}

\begin{proof}
Let $\{U_i\}_{i\in I}$ be an open cover of $\bar{G}$. The topology on $\partial G$ is the same as the standard one because the topology described is easily seen to be the one induced by the visual metric. Since $\partial G$ is compact there exists a finite set $\{U_i\}_{i \in J}$ covering $\partial G$. For each $i \in J$ there exists $\xi_i \in \partial G$ and $M_i > 0$ such that $ \{x \in \bar{G}|(x,\xi_i) > M_i\} \subseteq U_i$. If now $g \notin U_i$ and $\eta \in U_i$ then $M_i \gtrsim (g,\eta)$ by hyperbolicity. Therefore if $g \notin U_i$ for all $i \in J$ then there exists $M>0$ such that $(g,\eta) < M$ for all $\eta \in \partial G$. Now by \ref{nonempty shadow} $|g| \approx \sup_{\xi \in \partial G} (g,\xi)$ so $G \backslash \bigcup_{i\in J}U_i$ is bounded in the metric $d$ and therefore has compact closure in $G$. Thus the closure of $G \backslash \bigcup_{i\in J}U_i$ can be covered by finitely many of the remaining $U_i$ as needed.

To see that the topology is independent of $d$ notice that if $\xi \in \partial G$ then $g_n \to \xi$ if and only if $(g_n,\xi) \to \infty$ and by \ref{boundary invariance} this condition is independent of $d$.
\end{proof}

By \cite[Theorem~6.1.4]{DSU17},\cite[Section~3]{CCMT15} elements in $G$ can be classified into three types:

\begin{enumerate}
    \item Elliptic elements: elements with bounded orbits.
    \item Parabolic elements: non-elliptic elements with one fixed point in $\partial G$.
    \item Hyperbolic elements: non-elliptic elements with two fixed points in $\partial G$. One of the fixed points is attracting and the other is repelling. If $\xi \in \bar{G}$ is not the repelling fixed point of a hyperbolic element $g$ then $g^n \xi$ converges to the attracting fixed point.
\end{enumerate}

In the discrete case there are no parabolic elements but in the locally compact case they can appear.

By \cite[Section~3]{DSU17} \cite[Theorem~6.8]{SA96} locally compact hyperbolic groups can be classified in to three types:

\begin{enumerate}
    \item Elementary: groups with finite boundary. In this case either $\partial G = \phi$ and $G$ is compact or $|\partial G| = 2$ and $G$ contains an infinite cyclic co-compact subgroup.
    \item Non-elementary amenable: non-elementary groups stabilizing a point in $\partial G$. These groups have uncountable boundary and are amenable by \cite[Theorem 6.8]{SA96}.
    \item General type: groups not stabilizing any point in $\partial G$. These groups have uncountable boundary and are never amenable.
\end{enumerate}

Being of general type is equivalent to being non-amenable. By \cite[Theorem~7.3]{CCMT15} if a non-elementary hyperbolic locally compact group is amenable then it is never unimodular. Thus unimodular locally compact hyperbolic groups are either elementary or of general type. 

Non-elementary amenable groups do not exist in the discrete case but in the locally compact case they can appear. Examples of such groups are minimal parabolic subgroups of rank-$1$ simple Lie groups with finite center.

\begin{theorem}
If $G$ is hyperbolic locally compact of general type then the action on $\partial G$ is minimal.
\end{theorem}

\begin{proof}
By \cite[Theorem~7.4.1]{DSU17} any closed invariant subset of $\partial G$ containing at least two points is all of $\partial G$ but since $G$ is of general type there is no fixed point in $\partial G$.
\end{proof}

Finally we will need the following lemma:

\begin{lemma} \label{no automorphisems}
Let $G$ be hyperbolic locally compact of general type. Any $G$-equivariant continuous map $f:\partial G \to \partial G$ is trivial.
\end{lemma}

\begin{proof}
Since $G$ is of general type, $G$ contains hyperbolic elements (\cite[Section~3]{CCMT15}). The set of all fixed points in $\partial G$ of hyperbolic elements is invariant and non-empty so since the action is minimal it is dense. If $f:\partial G \to \partial G$ is $G$-equivariant and continuous then $f$ must stabilize the attracting fixed point and the repelling fixed point of any hyperbolic element, since continuity and equivariance imply that if $\xi$ is attracting or repelling for $g$ so is $f(\xi)$. Thus  $f$ fixes a dense set.
\end{proof}

We add one last remark on hyperbolic locally compact groups. It is not clear to us that given $d \in \mathcal{D}(G)$ the corresponding Gromov product $(\cdot,\cdot) : \bar{G} \times \bar{G} \to \R$ is measurable, however one can always slightly modify it to be measurable. We will only be interested in the values of the Gromov product up to a uniform constant since we are only interested in the coarse geometric structure of $d$. Therefore we will change the definition of the Gromov product in a way that changes the values only by a bounded amount and ensures that it is measurable. Choose a countable subset $S \subseteq G$ which is co-bounded, i.e any element is a uniformly bounded distance away from an element of $S$. The space $(S,d)$ is hyperbolic and roughly isometric to $(G,d)$ so they give rise to the same boundary. Since the Gromov product is lower semi-continuous on $\bar{S}$ it is Borel measurable (this is also true on $G$ but the Borel $\sigma$-algebra will be the one corresponding to the metric $d$ and not to the standard topology on $G$, this is the reason for this entire remark). Denote the Gromov product on $\bar{S}$ by $\langle \cdot,\cdot \rangle$. Choose a measurable closest point projection $P: G \to S$ and extend it to the boundary via the identity. Now if $x,y,z \in \bar{G}$ the function $\langle P(x), P(y) \rangle_{P(z)}$ is Borel measurable on $G$ and at bounded distance from the original definition of the Gromov product on $G$, so we can take it to be the new Gromov product. We will make one more change for convenience: we assume that the formula $(g,h)_o = \frac{1}{2}( |g| + |h| -d(g,h))$  still holds when $g,h,o \in G$.

\subsection{Non-singular actions and Koopman representations}

Consider a locally compact group $G$ with a measurable non-singular (i.e. measure class preserving) action on a measure space $(X,\mu)$. One can define an associated unitary representation of $G$ called the Koopman representation on $L^2(X,\mu)$ by:
$$[\pi(g)f](\xi) = \sqrt{\frac{dg_*\mu}{d\mu}(\xi)}f(g^{-1}\xi)$$

The Koopman representation depends only on the measure class $[\mu]$ and not on the actual measure. If $\nu$ is in the same measure class as $\mu$ then the map $L^2(X,\nu) \to L^2(X,\mu)$ defined by $f \mapsto \sqrt{\frac{d\mu}{d\nu}}f$ is an isomorphism of the Koopman representations corresponding to $\mu,\nu$.

By \cite[Proposition~A.6.1]{BDV08}, if $G$ is $\sigma$-compact and $(X,\mu)$ is a standard measure space this representation is continuous with respect to the strong operator topology. If $G$ is locally compact hyperbolic then $G$ is compactly generated hence $\sigma$-compact so the above theorem holds.

\section{Patterson-Sulivan Measures} \label{PSmeasures}

In this section we will construct from a metric $d \in \mathcal{D}(G)$ an invariant measure class $[\mu]$ on $\partial G$ called the Patterson-Sullivan measure class associated to $d$. The original construction of such measures in due to Patterson in the case of Fuchsian groups \cite{P76}. Sullivan later extended this to the case of discrete subgroups of $SO(n,1)$ \cite{DS79}. Further generalizations were made by Coornaert for discrete hyperbolic groups \cite{CO93}, and by Burger and Mozes for non-discrete groups acting on CAT(-1) spaces \cite{BM96}. We give a construction for general hyperbolic locally compact groups.

Fix a locally compact hyperbolic group $G$, $d \in \mathcal{D}(G)$, a left Haar measure $\lambda$ on $G$ and a visual metric $d_\varepsilon$.

\begin{definition}
    The critical exponent of (G,d) is 
    $$h = \inf \{s | \int\!e^{-s|g|}d\lambda(g) < \infty\} = \limsup_{r \to \infty}\frac{1}{r}\lambda(B_r(1)) $$
\end{definition}

$h$ is the exponential growth rate of $G$ with respect to $d$. If $G$ is non-elementary then by \cite[Section~3]{CCMT15} $G$ contains a discrete free 
 quasi isometrically embedded (Schottky) subgroup or subsemigroup and thus $h>0$. On the other hand, if $G$ is elementary then obviously $h = 0$. We denote $D = \frac{h}{\epsilon}$.

\begin{definition}
    A finite measure $\mu$ on $\partial G$ is called quasi-conformal of dimension $\frac{\alpha}{\epsilon}$ if for all $g \in G$, for $\mu$ almost every $\xi$ 
    $$\frac{dg_*\mu}{d\mu}(\xi) \asymp e^{-\alpha(|g|-2(g,\xi))}$$
\end{definition}

The quantity $|g|-2(g,\xi)$ in the above definition is the analogue of the Busemann function (see \cite[Definition II.8.17]{BH99}) corresponding to $\xi$, i.e. it roughly measures the "distance from $\xi$" normalized to be $0$ at $1$. We will not use this formally but it might help understand the geometric picture.

\begin{theorem}
If G is non-compact, there exists a quasi-conformal measure of dimension $D = \frac{h}{\epsilon}$ on $\partial G$.
\end{theorem}

We will later see that If G is non-elementary, the measure class of such a measure is uniquely determined. This measure class $[\mu]$ is also known as the Patterson-Sullivan measure class. After showing the measure class is unique we will fix such a measure $\mu$ and call it the Patterson-Sullivan measure of $(G,d)$.

\begin{proof}
There are many similar proofs in the literature. We imitate the proof in \cite[Theorem~5.4]{CO93}, another example is \cite[Section~1]{BM96}.
Recall that the space $\bar{G}$ has a natural compact topology on it described in \ref{compactification of G}. By \cite[Lemma~1.2]{BM96}
There exists an increasing continuous function $H:\R_+ \longrightarrow \R_+$ satisfying 

\begin{enumerate}
    \item $\int_G H(e^{|x|})e^{-s|x|}d\lambda(x)$ converges for $s > h$ and diverges at $s = h$.
    \item $\forall \alpha > 0, \exists t_0 > 0 $ such that $\forall t>t_0$ and $k>1$, $H(kt) \leq k^\alpha H(t)$.
\end{enumerate}

If one assumes that $\int e^{-h|x|}d\lambda(x)$ diverges one can take $H = 1$ and the proof becomes simpler. Denote
$$\mathcal{P}(s) = \int_G H(e^{|x|})e^{-s|x|}d\lambda(x)$$
and define probability measures on $\bar{G}$ by:
$$d\mu_s = \frac{1}{\mathcal{P}(s)}H(e^{|x|})e^{-s|x|}d\lambda(x)$$
Let $\mu$ be a weak-$*$ limit point of $\mu_s$ as $s \searrow h$. We claim $\mu$ is supported on $\partial G$ and is quasi-conformal of dimension $D$.

To see that $\mu(G) = 0$ notice that since $\mathcal{P}(s) \to \infty$ when $s \searrow h$, $\mu_s(K) \to 0$ for every compact $K \subset G$.

In order to show that $\mu$ is quasi-conformal of dimension $D$ it is enough to show that for every $g \in G$ and $\xi \in \partial G$ there exists a neighborhood $V$ of $\xi$ in $\bar{G}$ such that for any continuous function $f$ supported in $V$:
$$\int f(x) dg_*\mu(x) \asymp e^{-h(|g|-2(g,\xi))} \int f(x) d\mu(x)$$
uniformly in $g,\xi$. Indeed let $g \in G$, we have: 
$$\frac{dg_*\mu_s}{d\mu_s}(x) \asymp \frac{H(e^{|g^{-1}x|})}{H(e^{|x|})}e^{-s(|g^{-1}x| - |x|)} = \frac{H(e^{|g^{-1}x|})}{H(e^{|x|})}e^{-s(|g|-2(g,x))}$$ 
Let $\xi \in \partial G$. If $x \in G$ such that $(x,\xi) > |g|$ then by hyperbolicity $(g,x) \approx (g,\xi)$ uniformly in $g,\xi$. In addition since $| |x|-|g^{-1}x|| <|g|$ it follows from the second property of $H$ that there exists $C = C(|g|) > 0$ such that if $|x| > C$ then $\frac{H(e^{|g^{-1}x|})}{H(e^{|x|})} \asymp 1$. Let now $M = max\{C,|g|\}$ and consider $V = \{x \in \bar{G}|(x,\xi) > M\}$. By definition $V$ contains an open neighborhood of $\xi$. If $x\in V$ then in particular $|x| > M$. Therefore by the formula for $\frac{dg_*\mu_s}{d\mu_s}$ and since $(g,x) \approx (g,\xi)$ we get that for $x \in V$, $\frac{dg_*\mu_s}{d\mu_s}(x)  \asymp e^{-h(|g|-2(g,\xi))}$. Integrating any continuous function $f$ supported on $V$ we get:

$$\int f(x) dg_*\mu_s = \int f(x) \frac{dg_*\mu_s}{d\mu_s}(x)d\mu_s \asymp \int f(x) e^{-h(|g|-2(g,\xi))} d\mu_s$$

Taking a partial   limit $s \searrow h$ we get that the same holds for $\mu$ and thus $\mu$ is quasi-conformal of dimension $D$ as needed. 

\end{proof}

\begin{lemma}[Sullivan's shadow lemma] \label{shadow lemma}
If $\mu$ is quasi-conformal of dimension $D > 0$, is not supported on a single atom and $\sigma$ is large enough then 
$$\mu(\Sigma(g,\sigma)) \asymp_\sigma e^{-h|g|}$$
\end{lemma}

Note that if $\mu$ is supported on a single atom then it must be invariant and this can only happen if $h=0$ i.e. if $G$ is elementary.

\begin{proof}
Let $g \in G$, we have 

\begin{align*}
\mu(g^{-1}\Sigma(g,\sigma)) 
=&
\int_{\Sigma(g,\sigma)}\frac{dg_*\mu}{d\mu}(\xi) d\mu(\xi) 
\\
\asymp&
\int_{\Sigma(g,\sigma)} e^{-h(|g|-2(g,\xi))}d\mu(\xi) 
\\
\asymp&_\sigma
\int_{\Sigma(g,\sigma)} e^{h|g|} d\mu(\xi) 
\\
=&
\mu(\Sigma(g,\sigma))e^{h|g|}    
\end{align*}

Now in order to finish the proof we must bound $\mu(g^{-1}\Sigma(g,\sigma))$ independently of $g$.
Let $t<1$ be larger than the size of the largest atom of $\mu$. There exists $\delta > 0$ such that the measure of any ball of radius less then $\delta$ is less than $t$. Indeed, if $B_{r_n}(\xi_n)$ is a sequence of balls with $r_n \to 0$ and $\mu(B_{r_n}(\xi_n)) \geq t$ after passing to a convergent sub-sequence we can assume $\xi_n \to \xi$. For any $r>0$, $B_r(\xi)$ contains a ball of the form $B_{r_n}(\xi_n)$ so $\mu(B_r(\xi)) \geq t$ and thus $\mu(\{\xi\}) \geq t$ which is a contradiction.

Now the diameter of $\partial G \backslash g^{-1}\Sigma(g,\sigma)$ converges to $0$ as $\sigma \to \infty$ uniformly in g. To see this, notice that if $\xi \notin g^{-1}\Sigma(g,\sigma)$ then $g\xi \notin \Sigma(g,\sigma)$ so $(g,g\xi) \leq |g| - \sigma$. By lemma \ref{basic contraction lemma} we see that $(\xi,g^{-1}) \gtrsim \sigma$. So if $\xi, \eta \notin g^{-1}\Sigma(g,\sigma)$ then:
$$(\xi,\eta) \gtrsim min\{(\xi,g^{-1})(\eta,g^{-1})\} \gtrsim \sigma$$
Therefore $d_\varepsilon(\xi,\eta) \prec e^{-\sigma}$ as needed.
So for $\sigma$ large enough we have that $1 \geq \mu(g^{-1}\Sigma(g,\sigma)) > 1-t$ and thus $\mu(\Sigma(g,\sigma)) \asymp_\sigma e^{-h|g|}\mu(g^{-1}\Sigma(g,\sigma)) \asymp e^{-h|g|}$ finishing the proof.
\end{proof}

We assume from now on that $G$ is non-elementary so $h > 0$ and the shadow lemma holds.
Denote $A_R(\alpha) = \{g \in G |  R-\alpha < |g| < R+\alpha\}$, the annulus of radius $R$ and thickness $\alpha$ around the identity in $G$. For our purposes the parameter $\alpha$ should be thought of as an arbitrary parameter which we will take to be as large as we need and then keep it fixed.

\begin{lemma} \label{coverlemma}
For $\alpha$, $\sigma$ large enough and any $\xi \in \partial G$, $\lambda(\{g \in A_R(\alpha) | \xi \in \Sigma(g,\sigma) \}) \asymp_{\sigma,\alpha} 1$. i.e , The shadows of elements of the annuli $A_R(\alpha)$ cover $\partial G$ with "bounded multiplicity".
\end{lemma}

\begin{proof}
Let $\xi \in \partial G$ and $\gamma$ a roughly geodesic ray with $\gamma(0) = 1, \gamma(\infty) = \xi$. Fix $r > 0$ large enough so that the measure of an $r$-ball in G is positive, and choose $\alpha$ large enough so that $B_r(\gamma(R)) \subseteq A_R(\alpha)$. For any $g \in B_r(\gamma(R))$, $(g,\xi) \approx_r |g|$ and thus for $\sigma$ large enough, and independent of $g,\xi, R$  we have $\xi \in \Sigma(g,\sigma)$. Thus $\lambda(\{g \in A_R(\alpha) | \xi \in \Sigma(g,\sigma) \} \geq \lambda(B_r(\gamma(R)))$.

On the other hand if $g,h \in A_R(\alpha)$ and $\xi \in \Sigma(g,\sigma) \cap \Sigma(h,\sigma)$ then since $|g| \approx_\alpha |h| \approx_\alpha R$:
$$(g,h) \gtrsim_{\alpha, \sigma} min\{(g,\xi),(\xi,h)\} \approx_{\alpha, \sigma} R$$
Using again that $|g| \approx_\alpha |h| \approx_\alpha R$ and plugging this in to the definition of the Gromov product we deduce that $d(g,h) \approx_{\alpha, \sigma} 0$. Therefore $\{g \in A_R(\alpha) | \xi \in \Sigma(g,\sigma) \}$ is contained in a ball with radius in independent of $g,R,\xi$ giving the other inequality.
\end{proof}

We will also need the following discrete version of lemma \ref{coverlemma}:

\begin{lemma} \label{discrete aproximation}
Fix $C > 0$. For large enough $\alpha$, large enough $\sigma$ (perhaps depending on $C$) and any $R > 0$, there exists a finite subset $F \subseteq A_R(\alpha)$ such that the shadows $\{\Sigma(g,\sigma)| g\in F\}$ cover $\partial G$ and for any $g,h \in F$, $d(g,h) > C$.
\end{lemma}

\begin{proof}
Choose $\alpha$ large enough that for any roughly geodesic ray $\gamma$ with $\gamma(0) = 1$, $\gamma(R) \in A_R(\alpha)$. Since $\partial G$ is compact there exists a finite set of points $\xi_i \in \partial G$ such that the balls of radius $e^{-\epsilon R}$ cover $\partial G$. Choose roughly geodesic rays $\gamma_i$ such that $\gamma_i(0) = 1$ and $\gamma_i(\infty) = \xi_i$. By lemma \ref{shadow contains ball}, for some $\sigma > 0$ (depending on $C$), if $d(\gamma_i(R),\gamma_j(R)) \leq C$ then $B_{e^{-\epsilon R}}(\xi_j) \subseteq \Sigma(\gamma_i(R),\sigma)$. Let $F$ be a maximal subset of the $\gamma_i(R)$ with the property that for all $g,h \in F$, $d(g,h) > C$ then $\{\Sigma(g,\sigma)| g \in F\}$ covers $\partial G$ and has the desired properties. 

\end{proof}

The only use we will make of the parameter $\sigma$ is in lemma \ref{shadow contains ball}, lemma \ref{shadow lemma}, lemma \ref{coverlemma} and in lemma \ref{discrete aproximation}, furthermore we will only use lemma \ref{shadow contains ball} and lemma \ref{discrete aproximation} for a specific constant $C$ in the proof of lemma \ref{uniformly bounded} depending only on $(G,d)$. All three lemmas only hold for $\sigma$ which is large enough. We fix such a $\sigma$ and denote $\Sigma(g) = \Sigma(g,\sigma)$ for every $g \in G$. After this we will have no more need to mention the parameter $\sigma$ explicitly again. We state what we will use from lemma \ref{shadow contains ball}, lemma \ref{shadow lemma}, lemma \ref{coverlemma} and lemma \ref{discrete aproximation}:

\begin{itemize}
    \item For every $g \in G$, $R > 0$ and roughly geodesic ray $\gamma$ with $\gamma(0) = 1, \gamma(\infty) = \xi$: $B_{e^{-\epsilon R}}(\xi) \subseteq \Sigma(g)$.
    \item For every $g \in G$: $\mu(\Sigma(g)) \asymp e^{-h|g|}$.
    \item For large enough $\alpha$ and any $\xi \in \partial G$, $\lambda(\{g \in A_R(\alpha) | \xi \in \Sigma(g) \}) \asymp_\alpha 1$.
    \item For a fixed constant $C$ from the proof of \ref{uniformly bounded} which depends only on $(G,d)$, for large enough $\alpha$ and for any $R > 0$ there exists a finite $F \subseteq A_R(\alpha)$ such that the shadows $\{\Sigma(g)| g\in F\}$ cover $\partial G$ and for any $g,h \in F$, $d(g,h) > C$.
\end{itemize}

An immediate consequence of the shadow lemma (or alternatively lemma \ref{basic contraction lemma}) is that for all $g$, $\Sigma(g) \neq \phi$. We fix $\hat{g} \in \Sigma(g)$ for every $g$ and denote $\check{g} = \widehat{g^{-1}}$. We have $(g,\hat{g}) \approx |g|$ uniformly in $g$. Perhaps the function $g \to \hat{g}$ can be chosen to be measurable but this will not affect us. Since $(g,\hat{g}) \approx |g|$ and $(g,x) \leq |g|$ for any $x$ we get that $(x,\hat{g}) \gtrsim min\{(x,g),(g,\hat{g})\} \approx (x,g)$. On the other hand, we get that $|g| \geq (g,x) \gtrsim min\{(g,\hat{g}),(\hat{g},x)\} \approx min\{|g|,(\hat{g},x)\}$ so:

\begin{equation} \label{passing_to_hat}
    (g,x) \approx min\{|g|,(\hat{g},x)\}
\end{equation}

\begin{lemma} \label{ahlfors regular}
Suppose $\mu$ is quasi-conformal of dimension $D$. The measure $\mu$ is Ahlfors regular of dimension $D = \frac{h}{\varepsilon}$ with respect to the metric $d_\varepsilon$, i.e for any $\xi \in \partial G$ and $\rho \leq diam(\partial G)$  

$$\mu(B_\rho(\xi)) \asymp \rho^D$$
\end{lemma}

\begin{proof}
Let $\xi, \rho$ be as above and let $\gamma : [0,\infty) \to G$ be a roughly geodesic ray with $\gamma(0) = 1, \gamma(\infty) = \xi$. We know that $B_\rho(\xi) \subseteq \Sigma(\gamma(\frac{-1}{\epsilon}ln(\rho)))$. Taking the measures of both sides we get by the shadow lemma that $\mu(B_\rho(\xi)) \prec  \rho^D$.

For the other inequality suppose $\eta \in \Sigma(\gamma(t))$, then $(\eta,\xi) \gtrsim min\{(\eta,\gamma(t)), (\gamma(t),\xi)\} \approx t$. Therefore $d_\varepsilon(\xi,\eta) \asymp e^{-\varepsilon(\xi,\eta)} \prec e^{-\varepsilon t}$, so there exists $C>0$ independent of $\xi, \rho$ such that if $t \geq \frac{-1}{\varepsilon}ln(\rho) + C$ then $d_\varepsilon(\xi, \eta) < \rho$. We conclude that $\Sigma(\gamma(\frac{-1}{\varepsilon}ln(\rho) + C)) \subseteq B_\rho(\xi)$. Taking the measures of these sets and using lemma \ref{shadow lemma} we get that $\rho^D \prec \mu(B_\rho(\xi))$.
\end{proof}

As a result, since $h \neq 0$, $\mu$ has no atoms. Because $\mu$ is Radon we also see that the measure class $[\mu]$ is determined uniquely by the assumption that $\mu$ is quasi-conformal of dimension $D$.

\begin{lemma} \label{PS are Hausdorff}
If $\mu$ is quasi-conformal of dimension $D$ then $\mu$ and the Hausdorff measure $H^D$ of dimension $D$ on $\partial G$ are in the same measure class and the corresponding Radon-Nikodym derivative is bounded and bounded away from $0$.
\end{lemma}

\begin{proof}
The proof is taken from \cite[Corollary~2.5.10]{DC13}. Let $A$ be a measurable set, and $U_i$ a countable cover of $A$ by sets of diameters $\rho_i < \delta$. Each $U_i$ is contained in a ball $B_i$ of radius $\rho_i$ so $\mu(A) \prec \sum_i \rho_i^D$. Taking  $\delta \to 0$ we get $\mu(A) \prec H^D(A)$ uniformly in A.

For the other direction, let $A$ be measurable. The measures $\mu$ and $H^D$ are Radon so for every 
$\epsilon > 0$ there exist compact $K$ and open $U$ with $K \subset A \subset U$ and $H^D(U\backslash K), \mu(U\backslash K) < \epsilon$. Since $K$ is compact there exists $\delta$ such that any ball centered at a point in $K$ of radius less than $\delta$ is contained in $U$. Let $U_i$ be a countable cover of $K$ with sets of diameters $\rho_i < \delta$, ordered such that $\rho_i$ is non-increasing. We can assume without loss of generality that $U_i \cap (K \backslash \bigcup_{j<i}U_j) \neq \phi$. Under this assumption each $U_i$ is contained in a ball $B_i$ of radius $\rho_i$ centered at a point in $K \backslash \bigcup_{j<i}U_j$. Therefore $B_i \subset U$ and if $C_i$ denote the balls radius $\frac{\rho_i}{2}$ centered at the same point as $B_i$ then $C_i$ are disjoint. Thus $\sum \rho_i^D \prec \mu(U)$ and taking $\delta$ to $0$ we get that $H^D(K) \prec \mu(U)$ uniformly in $A,\epsilon$. Now taking $\epsilon$ to $0$ we get $H^D(A) \prec \mu(A)$ uniformly in $A$.
\end{proof}

Since the Hausdorff measure class is independent of $\mu$ there is a unique measure class containing a quasi-conformal measure of dimension $D$. In addition since $\frac{d\mu}{dH^D} \asymp 1$, $H^D$ is quasi-conformal of dimension $D$. We now fix a quasi-conformal probability measure  $\mu$ of dimension $D$ and call it the Patterson Sullivan measure of $(G,d)$.
As a corollary we see that $\mu$ is ergodic. Indeed if, $E$ is a $G$ invariant subset and $\mu(E) \neq 0$ then the restriction of $\mu$ to $E$ is quasi-conformal of dimension $D$ and is thus equivalent to $\mu$ so E is co-null.

Finally we will need the following strengthening of the shadow lemma:

\begin{lemma} \label{generalized shadow lemma}
There exists a constant $0 < C$ such that for any $g \in G$ and any $s < |g| - C$:
$$\mu(\{ \xi \in \partial G | (g,\xi) > s \}) \asymp e^{-hs} $$
\end{lemma}

\begin{proof}
Let $s>0$ and set $\rho =e^{-\epsilon s}$. If $(\xi,g) > s$ then by estimate \ref{passing_to_hat}, $(\hat{g},\xi) \gtrsim s$. Thus by estimate \ref{visual_metric} there exists $L > 0$, independent of $g,s$, such that $\{ \xi \in \partial G | (g,\xi) > s \} \subset B_{L\rho}(\hat{g})$. Using Ahlfors regularity we get one inequality. In the other direction, by estimate \ref{passing_to_hat} there exists $C > 0$, independent of $g, \xi$, such that if $(\hat{g},\xi) > s + C$ then $(g,\xi) > min\{s , |g| - C\}$, so if $s < |g| - C$ then $(g,\xi) > s$. Thus by estimate \ref{visual_metric} there exists $L > 0$, independent of $g,s$,  satisfying $B_{\frac{\rho}{L}}(\hat{g}) \subset \{ \xi \in \partial G | (g,\xi) > s \}$. The second inequality follows again by Ahlfors regularity.
\end{proof}

\section{Some Growth Estimates} \label{growth setimates}

In this section we will use the shadow lemma to obtain growth estimates on the group $G$. We generalize many of the results of Garncarek (\cite[Section~4]{G14}) to the locally compact case, many of the proofs are similar. The main difference occurs in lemma \ref{cancellation lemma} which requires the assumption of unimodularity which holds trivially for discrete groups. We keep the notation of the previous section. 

First we deduce the following corollary from lemma \ref{coverlemma}. 

\begin{theorem} \label{growth estimate}
If $G$ is non-elementary and $\alpha$ is large enough, $\lambda(A_R(\alpha)) \asymp_\alpha \lambda(B_R(1)) \asymp e^{hR}$.
\end{theorem}

\begin{proof}
The proof is essentially a double counting argument. Suppose $\alpha$ is large enough for lemma \ref{coverlemma} to hold. Consider the set $A = \{(g,\xi)| \xi \in \Sigma(g)\}$. Using Fubini's theorem we see that on the one hand by lemma \ref{coverlemma}:
$$\lambda \times \mu (A) = \int \lambda(\{g \in A_R(\alpha)| \xi \in \Sigma(g)\}) d\mu(\xi) \asymp_\alpha \int 1 d\mu(\xi) = 1$$

On the other hand by lemma \ref{shadow lemma}:
$$\lambda \times \mu (A) = \int_{A_R(\alpha)} \mu(\Sigma(g))d\lambda(g)  \asymp_\alpha \int_{A_R(\alpha)} e^{-hR}d\lambda(g) = \lambda(A_R(\alpha))e^{-hR}$$
 
 Comparing the two we get the result for annuli. The result for balls follows by covering the ball with annuli and using the formula for the sum of a geometric series.
\end{proof}

We will need one more lemma estimating the size of certain subsets of annuli.

\begin{lemma} \label{cone lemma}
For any $ s,R,\alpha  > 0$ and $\xi \in \partial G$:
$$\lambda(\{g \in A_R(\alpha)| (g,\xi) > s\}) \prec_\alpha e^{h(R-s)}$$
\end{lemma}

\begin{proof}
Let $\gamma$ be a roughly geodesic ray with $\gamma(0) = 1, \gamma(\infty) = \xi$ and $g$ as above. Note that $(\gamma(s),\xi) \approx s$, so $(g,\gamma(s)) \gtrsim min\{(g,\xi),(\xi,\gamma(s))\} \approx s$, but $|\gamma(s)| \approx s$ so $s \gtrsim (g,\gamma(s))$ and therefore $(g,\gamma(s)) \approx s$. Opening up the definition of the Gromov product and plugging in $|g| \approx_\alpha R, |\gamma(s)| \approx s$ we get $d(g,\gamma) \approx R-s$. Thus $\{g \in A_R(\alpha)| (g,\xi) > s\}$ is contained in a ball of radius $r \approx R-s$ centred at $\gamma(s)$ . Applying lemma \ref{growth estimate} we obtain the result.
\end{proof}

The following lemma is one of the key geometric ideas in the proof of irreducibility of boundary representations. It generalizes the fact that given any two words $w_1,w_2$ in a free group one can find a generator $x$ such that the product $w_1 x w_2$ has no cancellation. Interestingly the proof requires the added assumption of unimodularity, which holds trivially in the discrete case.

\begin{lemma} \label{cancellation lemma}
Suppose G is unimodular and non-elementary. There exists $\tau$ such that for all $g,h \in G$, $\mu (\{g' \in B_\tau(g) | |g'| + |h| - 2\tau \leq |g'h|  \}) \asymp 1$.
\end{lemma}

\begin{proof}
Let $g,h \in G$, and for every $s \in G$ fix a roughly geodesic segment $\gamma_s:[0,|s|] \rightarrow G$ in G with $\gamma_s(0) = 1, \gamma_s(|s|) = s$. Let $g' \in B_\tau(g)$ such that $|g'| + |h| - 2\tau > |g'h|$, this is equivalent to $(g'h,1)_{g'} > \tau$. We get:
$$(\gamma_{g'}(|g'| -\tau),g'\gamma_h(\tau))_{g'} \gtrsim min\{(\gamma_g'(|g'| -\tau),1)_{g'}, (1,g'h)_{g'}, (g'h,g'\gamma_h(\tau))_{g'}\} \approx \tau$$
Since $d(g', \gamma_{g'}(|g'| -\tau)),d(g', g'\gamma_h(\tau)) \approx \tau$, we conclude that $d(\gamma_{g'}(|g'| -\tau), g'\gamma_h(\tau)) \approx 0$.
Similarly we have that:
$$(\gamma_g(|g'|-\tau),\gamma_{g'}(|g'| -\tau)) \gtrsim min \{(\gamma_g(|g'|-\tau),g),(g,g'),(g',\gamma_{g'}(|g'| -\tau)) \} \gtrsim |g'| - \tau$$
So since $|\gamma_{g'}(|g'| -\tau)|, |\gamma_g(|g'|-\tau)| \approx |g'| - \tau$ we conclude that $d(\gamma_{g'}(|g'| -\tau),\gamma_g(|g'|-\tau)) \approx 0$.

Putting everything together we get that $d(g'\gamma_h(\tau),\gamma_g(|g'|-\tau)) \approx 0$. Notice that $\gamma_g(|g'|-\tau) \in \gamma_g([|g|-2\tau,|g|])$, so if  $|g'| + |h| - 2\tau > |g'h|$, then $g'\gamma_h(\tau)$ is in a tubular neighborhood of constant distance (independent of $\tau$) from the roughly geodesic segment $\gamma_g([|g|-2\tau,|g|])$. Denote this neighborhood by $N(\gamma_g([|g|-2\tau,|g|]))$.

Consider the map $g' \mapsto g'\gamma_h(\tau)$. This map is measure preserving since $G$ is unimodular. Because the measure of $B_\tau(g)$ grows exponentially with $\tau$ and the measure of $N(\gamma_g([|g|-2\tau,|g|]))$ grows linearly with $\tau$, for $\tau$ large enough this map must send more than half of the measure of $B_\tau(g)$ outside of $N(\gamma_g([|g|-2\tau,|g|]))$. The elements not being sent to $N(\gamma_g([|g|-2\tau,|g|]))$ satisfy $|g'| + |h| - 2\tau \leq |g'h|$.
\end{proof}

\section{Operators on $L^2(\mu)$} \label{constructing operators}

 In this section we use the tools we have developed to construct operators in the von Neuman algebra of the representation $\pi$. These operators will later be used to  show irreducibility of the representation. The results generalize \cite[Section~5]{G14} and many of the core ideas come from there, although some of the constructions become more subtle and need to be changed.

\subsection{Estimates on $\pi(g)$}

To simplify calculations we assume without loss of generality that $D > 1$. Denote $P_g(\xi) = \sqrt{\frac{dg_*\mu}{d\mu}}(\xi)$, $\widetilde{P}_g = \frac{P_g}{\|P_g\|_1}$ and $\widetilde{\pi}(g) = \frac{\pi(g)}{\|P_g\|_1}$. Note that $\|P_g\|_1 = \langle \pi(g)1,1 \rangle = \langle 1,\pi(g^{-1})1 \rangle = \|P_{g^{-1}}\|_1$ so that $\widetilde{\pi}(g)^* = \widetilde{\pi}(g^{-1})$. We will call the weak operator closure of the image of the positive $L^1$ functions under $\pi$ the \textbf{positive cone} of $\pi$.

We will use the following lemma several times during calculations.

\begin{lemma}
Let $(X,\mu)$ be a $\sigma$-finite measure space and $f$ a positive measurable function on $X$ then:

$$\int_X f(\xi) d\mu = \int \limits_0^\infty \mu(\{\xi|f(\xi) > t\})dt$$
\end{lemma}

\begin{proof}
The proof follows from Fubini's theorem:

$$\int_X f(\xi) d\mu = \int_X \int \limits_0^\infty \mathbbm{1}_{\{(t,\xi)|t < f(\xi)\}} dt d\mu = \int \limits_0^\infty \int_X \mathbbm{1}_{\{(t,\xi)|t < f(\xi)\}}d\mu dt = \int \limits_0^\infty \mu(\{\xi|f(\xi) > t\})dt $$
\end{proof}

\begin{lemma}
The following estimates are satisfied uniformly in $g, \xi$:
$$\|P_g\|_1 \asymp e^{-\frac{h|g|}{2}}(1+|g|)$$
$$\widetilde{P}_g \asymp \frac{e^{h(g,\xi)}}{1+|g|}$$

\end{lemma}

\begin{proof}

By definition $P_g(\xi) \asymp e^{h((g,\xi) - \frac{|g|}{2})}$. Let $C$ be the constant from lemma \ref{generalized shadow lemma}. If $|g| < C$ then $\|P_g\|_1$ is bounded by a constant. Assume $|g| \geq C$. By lemma \ref{generalized shadow lemma} $\mu(\{\xi \in \partial G | (g,\xi) > \frac{ln(t)}{h}\}) \asymp \frac{1}{t}$ for $t < |g|-C$. we get:

\begin{align*}
    e^{\frac{|g|}{2}}\int_{\partial G}  P_g(\xi) d\mu &\asymp \int_{\partial G} e^{h(g,\xi)} d\mu 
    \\& = \int_0^\infty \mu(\{\xi \in \partial G | (g,\xi) > \frac{ln(t)}{h}\})dt
    \\& \asymp 1 + \int_1^{e^{h(|g| - C)}} \mu(\{\xi \in \partial G | (g,\xi) > \frac{ln(t)}{h}\})dt
    \\& \asymp 1 + \int_1^{e^{h(|g| - C)}} \frac{dt}{t} 
    \\& = 1 + h(|g| - C) 
    \\& \asymp 1 + |g|
\end{align*}

Which gives the first estimate. The second estimate now follows from the first and from $P_g(\xi) \asymp e^{h((g,\xi) - \frac{|g|}{2})}$.

\end{proof}

\begin{lemma} \label{bounded at 1}
For every $R$,$\alpha$ and for almost every $\xi \in \partial G$
$$\int_{A_R(\alpha)} \widetilde{P}_g(\xi)d\lambda(g) \prec_\alpha e^{hR}$$
\end{lemma}

\begin{proof}
Using theorem \ref{growth estimate}, lemma \ref{cone lemma} and the previous lemma we see that:
\begin{align*}
   & (1+R)\int_{A_R(\alpha)} \widetilde{P}_g(\xi) d\lambda
    \\& \asymp_\alpha \int_{A_R(\alpha)} e^{h(g,\xi)} d\lambda
    \\& = \int_0^{\infty} \lambda\{g \in A_R(\alpha)| (g,\xi) >\frac{ln(t)}{h} \} dt
    \\& \asymp_\alpha e^{hR} + \int_1^{e^{h(R+\alpha)}} \lambda\{g \in A_R(\alpha)| (g,\xi) >\frac{ln(t)}{h} \} dt
    \\& \prec_\alpha e^{hR} + \int_1^{e^{h(R+\alpha)}} e^{h(R-\frac{ln(t)}{h})} dt 
    \\& = e^{hR}(1+R+\alpha)
\end{align*}
 Giving the required estimate.
\end{proof}

We now obtain estimates on matrix coefficients of the representation $\pi$. Denote the space of $d_\varepsilon$-Lipschitz functions on $\partial G$ by $Lip(\partial G)$, and for $\phi \in Lip(\partial G)$ denote the Lipschitz constant by $L(\phi)$. $Lip(\partial G)$ is dense in $L^2(\mu)$. To see this notice that  by the Lebesgue differentiation theorem, which holds for Ahlfors regular measures by \cite[Theorem~1.8]{HJ01}, every non-zero $L^2$ function has non-zero inner product with the characteristic function of some ball. Thus characteristic functions of balls span a dense subspace of $L^2(\mu)$. Since these can be approximated by Lipschitz functions, $Lip(\partial G)$ is dense in $L^2(\mu)$.

\begin{lemma} \label{matrix coefficient esimate}
Let $g \in G$. For any $\phi, \psi \in Lip(\partial G)$ we have:

$$|\langle\widetilde{\pi}(g)\phi,\psi\rangle - \phi(\check{g})\overline{\psi(\hat{g})}| \prec \frac{L(\phi)\|\psi\|_\infty + L(\psi)\|\phi\|_\infty}{(1+|g|)^{\frac{1}{D}}}$$
\end{lemma}

\begin{proof}

By the choice of normalization $\langle\widetilde{\pi}(g)1,1\rangle = 1$. Therefore:

\begin{align*}
&|\langle\widetilde{\pi}(g)\phi,\psi\rangle - \phi(\check{g})\overline{\psi(\hat{g})}| 
\\
= &
|\langle\widetilde{\pi}(g)\phi,\psi\rangle - \langle\widetilde{\pi}(g)\phi(\check{g}),{\psi(\hat{g})}\rangle| 
\\
\leq &
|\langle\widetilde{\pi}(g)\phi,\psi - \psi(\hat{g})\rangle| + |\langle\widetilde{\pi}(g)(\phi - \phi(\check{g})), \psi(\hat{g})\rangle|
\\
= &
|\langle\widetilde{\pi}(g)\phi,\psi - \psi(\hat{g})\rangle| + |\langle\phi - \phi(\check{g}), \widetilde{\pi}(g^{-1})\psi(\hat{g})\rangle|
\end{align*}

We estimate the first term, the second term can be estimated similarly after replacing $g$ by $g^{-1}$. Let $\rho > 0$ and consider the ball $B_\rho(\hat{g})$. We will decompose the integral over $ B_\rho(\hat{g})$ and the complement.

\begin{align*}
&|\langle\widetilde{\pi}(g)\phi,\psi - \psi(\hat{g})\rangle| 
\\
\leq &
\int_{\partial G} \widetilde{P}_g(\xi) |\phi(g^{-1}\xi)| |\psi(\xi) - \psi(\hat{g})|d\mu(\xi) 
\\
=&
\int_{B_\rho(\hat{g})} \widetilde{P}_g(\xi) |\phi(g^{-1}\xi)| |\psi(\xi) - \psi(\hat{g})|d\mu(\xi) +
\\&
\int_{\partial G \backslash B_\rho(\hat{g})} \widetilde{P}_g(\xi) |\phi(g^{-1}\xi)| |\psi(\xi) - \psi(\hat{g})|d\mu(\xi)
\end{align*}

For $ \xi \in B_\rho(\hat{g})$, $|\psi(\xi) - \psi(\hat{g})| \leq L(\psi)\rho$ so the first integral is bounded by $\|\phi\|_\infty L(\psi) \rho$. In order to estimate the second integral recall that for any $\xi$,  $(\hat{g},\xi) \geq (g,\xi)$, so that
$\widetilde{P}_g \asymp \frac{e^{h(g,\xi)}}{1+|g|} \prec \frac{e^{h(\hat{g},\xi)}}{1+|g|} \asymp \frac{d_\varepsilon(\hat{g},\xi)^{-D
}}{1+|g|}$. Thus (using the assumption $D>1$):

\begin{align*}
&\int_{\partial G \backslash B_\rho(\hat{g})} \widetilde{P}_g(\xi) |\phi(g^{-1}\xi)| |\psi(\xi) - \psi(\hat{g})|d\mu(\xi) 
\\
\prec& 
\int_{\partial G \backslash B_\rho(\hat{g})} \frac{\|\phi\|_\infty L(\psi) d_\varepsilon(\hat{g},\xi)^{1-D}}{1+|g|}d\mu(\xi) 
\\
\leq&
\int_{\partial G \backslash B_\rho(\hat{g})} \frac{\|\phi\|_\infty L(\psi) \rho^{1-D}}{1+|g|}d\mu(\xi) \leq 
\frac{\|\phi\|_\infty L(\psi) \rho^{1-D}}{1+|g|}
\end{align*}

All in all we get a bound on the first term proportional to $\|\phi\|_\infty L(\psi)(\rho + \frac{\rho^{1-D}}{1+|g|})$ and similarly a bound on the second term proportional to $\|\psi\|_\infty L(\phi)(\rho + \frac{\rho^{1-D}}{1+|g|})$. Finally: 
$$|\langle\widetilde{\pi}(g)\phi,\psi\rangle - \phi(\check{g})\overline{\psi(\hat{g})}| \prec(\|\phi\|_\infty L(\psi) + \|\psi\|_\infty L(\phi))(\rho + \frac{\rho^{1-D}}{1+|g|})$$

One can now differentiate with respect to $\rho$ in order to find the best bound. To obtain the claim it is enough to take $\rho = (1+|g|)^{-\frac{ 1}{D}}$.

\end{proof}

\subsection{Kernel operators}

Given a function $K \in L^\infty (\partial^2 G)$ we can define a bounded operator $T_K$ on $L^2(\mu)$ by $T_Kf(\xi) = \int K(\xi,\eta) f(\eta) d\mu(\eta)$. These operators are called \textbf{kernel operators}. The key theorem is that kernel operators with non-negative kernel $K$ are in the positive cone of $\pi$. This will give us a rich enough family of operators to show the representation is irreducible.

Assume now that $G$ is unimodular so lemma \ref{cancellation lemma} holds and fix  and a positive kernal $K \geq 0$ in $L^\infty(\partial^2 G)$. Given $R$, consider a family $F \subset A_R(\alpha)$ as in lemma \ref{discrete aproximation}, for some fixed large enough $\alpha$. Fix $\tau$ as in lemma \ref{cancellation lemma} and denote $B = B_\tau(1)$. We partition $\partial^2 G$ as follows, choose a linear order on $F^2$, the sets $\Sigma(g) \times \Sigma(h)$ cover $\partial^2 G$ for $g,h \in F$. Using the order on $F^2$ we define $V_{g,h} = \Sigma(g) \times \Sigma(h) \backslash \bigcup_{(s,t) < (g,h)} \Sigma(s) \times \Sigma(t)$. After discarding empty sets $\{V_{g,h}| (g,h) \in F^2\}$ is a partition of $\partial^2 G$. Denote $U_{g,h} = \{k \in gBh^{-1}| |k| > |g| + |h| -3\tau\}$, by lemma \ref{cancellation lemma} and since $G$ is unimodular $\lambda(U_{g,h}) \asymp 1$. Note that $U_{g,h} \subseteq A_{2R}(2\alpha+3\tau)$.
Define:

$$w_{g,h} =\frac{ \int_{V_{g,h}} K d\mu^2}{\lambda(U_{g,h})}$$
Denote also $w(k) = \sum_{F^2} w_{g,h}\mathbbm{1}_{U_{g,h}}(k)$ and $S_{R,g,h} = \int_{U_{g,h}} \widetilde{\pi}(k) d\lambda(k)$.
Finaly define:

$$S_R = \int_G w(k)\widetilde{\pi}(k) d\lambda(k) = \sum_{F^2}w_{g,h}\int_{U_{g,h}} \widetilde{\pi}(k) d\lambda(k) = \sum_{F^2} w_{g,h} S_{R,g,h}$$

We now prove that if $G$ is a non-elementary unimodular locally compact hyperbolic group then $S_R \underset{R \to \infty}{\longrightarrow} T_K$ in the weak operator topology. This follows easily from the next two lemmas.

\begin{lemma}
If $G$ is a non-elementary unimodular locally compact hyperbolic group and $\phi, \psi \in Lip(\partial G)$ then $\langle S_R \phi,\psi \rangle \underset{R \to \infty}{\longrightarrow} \langle T_K \phi, \psi \rangle $.
\end{lemma}

\begin{proof}
 If $k \in U_{g,h}$ then $|k| \approx |g| + |h|$, since $d(g,k) \lesssim |h|$ we get that $(g,k) \approx |g|$. Thus we have $(\hat{g},\hat{k}) \gtrsim min\{(\hat{g},g)(g,k)(k,\hat{k})\} \approx |g| \approx R$ so $d_\varepsilon(\hat{g}, \hat{k}) \prec e^{-\varepsilon R}$. Similarly $d_\varepsilon(\hat{h}, \check{k}) \prec e^{-\varepsilon R}$. Therefore $|\phi(\hat{h})\overline{\psi(\hat{g})} - \phi(\check{k})\overline{\psi(\hat{k})}| \prec (L(\phi)\|\psi\|_\infty + L(\psi)\|\phi\|_\infty)e^{-\varepsilon R}$, so by lemma \ref{matrix coefficient esimate}:

\begin{align*}
    &|\langle\widetilde{\pi}(k)\phi,\psi\rangle - \phi(\hat{h})\overline{\psi(\hat{g})}| 
    \\
    <& |\langle\widetilde{\pi}(k)\phi,\psi\rangle - \phi(\check{k})\overline{\psi(\hat{k})} | + |\phi(\check{k})\overline{\psi(\hat{k})} - \phi(\hat{h})\overline{\psi(\hat{g})}| 
    \\
    \prec&
    (L(\phi)\|\psi\|_\infty + L(\psi)\|\phi\|_\infty) \left(\frac{1}{(1+R)^{\frac{1}{D}}} + e^{-\varepsilon R} \right )
\end{align*}

As a result we get:

$$\abs{ \frac {1}{\lambda(U_{g,h})}\left \langle S_{R,g,h} \phi, \psi \right \rangle  - \phi(\hat{h})\overline{\psi(\hat{g})}} \prec (L(\phi)\|\psi\|_\infty + L(\psi)\|\phi\|_\infty) \left (\frac{1}{(1+R)^{\frac{1}{D}}} + e^{-\varepsilon R} \right ) $$

In addition, if $\xi,\eta \in V_{g,h}$ then $d_\varepsilon(\xi,\hat{g}), d_\varepsilon(\eta,\hat{h}) \prec e^{-\varepsilon R}$ so  $|\phi(\hat{h})\overline{\psi(\hat{g})} - \phi(\eta)\overline{\psi(\xi)}| \prec (L(\phi)\|\psi\|_\infty + L(\psi)\|\phi\|_\infty)e^{-\varepsilon R}$ and thus:

\begin{align*}
&\int_{V_{g,h}} K(\xi,\eta)|\phi(\eta)\overline{\psi(\xi)} - \phi(\hat{h})\overline{\psi(\hat{g})}|d\mu^2(\eta,\xi)
\\
\prec&
\int_{V_{g,h}} K(\xi,\eta)d\mu^2 (L(\phi)\|\psi\|_\infty + L(\psi)\|\phi\|_\infty)e^{-\varepsilon R}
\end{align*}

Finally, using the previous two estimates:
\begin{align*}
&|\langle S_R \phi, \psi \rangle - \langle T_K \phi, \psi \rangle| 
\\
\leq& 
 \sum_{F^2}  \abs{ w_{g,h} \langle S_{R,g,h} \phi ,\psi \rangle - \int_{V_{g,h}} K(\xi,\eta)\phi(\eta)\overline{\psi(\xi)}d\mu^2(\xi,\eta)} 
 \\
 \leq&
 \sum_{F^2}\left(\int_{V_{g,h}} K d\mu^2\right)\abs{\frac{1}{\lambda(U_{g,h})} \langle S_{R,g,h} \phi ,\psi \rangle -  \phi(\hat{h})\overline{\psi(\hat{g})}} 
 \\
 +&
 \sum_{F^2}\int_{V_{g,h}} K(\xi,\eta)| \phi(\hat{h})\overline{\psi(\hat{g})} - \phi(\eta)\overline{\psi(\xi)}|d\mu^2(\eta,\xi)
 \\
 \prec&
\sum_{F^2}\left[ \left(\int_{V_{g,h}} K d\mu^2 \right) (L(\phi)\|\psi\|_\infty + L(\psi)\|\phi\|_\infty)\left(\frac{1}{(1+R)^{\frac{1}{D}}} + 2e^{-\varepsilon R} \right )\right]
\\
= &
 \|K\|_1(L(\phi)\|\psi\|_\infty + L(\psi)\|\phi\|_\infty)\left(\frac{1}{(1+R)^{\frac{1}{D}}} + 2e^{-\varepsilon R} \right )
 \end{align*}

The right hand side tends to $0$ as $R \to \infty$  as needed.

\end{proof}

\begin{lemma} \label{uniformly bounded}
If $G$ is a non-elementary unimodular locally compact hyperbolic group then the norms $\|S_R\|_{op}$ are bounded uniformly in R.
\end{lemma}

\begin{proof}
We first show $w(x) \prec e^{-2hR}$. For every $(g,h) \in F^2$, $V_{g,h} \subseteq \Sigma(g) \times \Sigma(h)$, so by lemma \ref{shadow lemma} $\int_{V_{g,h}}Kd\mu^2 \prec \|K\|_\infty e^{-2hR} \prec e^{-2hR}$. Since $\lambda(U_{g,h}) \asymp 1$ we get that $w_{g,h} \prec e^{-2hR}$. Now all we need to show is that the $U_{g,h}$ are disjoint. Indeed if $U_{g,h} \cap U_{g',h'} \neq \phi$ there exist $s \in B_\tau(g), s' \in B_\tau(g')$ such that $|sh^{-1}| \geq |s| + |h| -4\tau, |s'h'^{-1}| \geq |s'| + |h'| -4\tau$ and $sh^{-1} = s'h'^{-1}$. $\tau$ is a fixed constant so $(sh^{-1},s) \gtrsim |s| \approx R, (s'h'^{-1},s')\gtrsim|s'|\approx R$. Since $|g|,|g'|,|s|,|s'| \approx R$ and $d(g,s),d(g',s') \approx 0$, we get $(g,s),(g',s') \approx R$. Thus using $sh^{-1} = s'h'^{-1}$, $(g,g') \gtrsim min\{(g,s),(s,sh^{-1})(s'h'^{-1},s'),(s',g')\} \approx R$. Since $|g|,|g'| \approx R$ we get that $d(g,g') \approx 0$. Taking inverses we similarly see that $d(h,h') \approx 0$. As a result there exists a constant $C$ depending only on $(G,d)$  such that if $U_{g,h} \cap U_{g',h'} \neq \phi$ then $d(g,g'), d(h,h') < C$. By the definition of $F$ (and our assumptions on the choice of $\sigma$ defining $\Sigma(g) = \Sigma(g,\sigma)$) we know that for this specific $C$, lemma \ref{discrete aproximation} holds and thus since $d(g,g'), d(h,h') < C$ we conclude that $g = g'$ and $h = h'$ as needed. Note that there is no circular reasoning in the proof since the constant $C$ which we obtained depends only on $(G,d)$ so our choice of the parameter $\sigma$ is independent of $R$.

Using $w(x) \prec e^{-2hR}$ and the fact that $w(k)$ is supported in $A_{2R}(2\alpha +3\tau)$ it follows from lemma \ref{bounded at 1} that $\|S_R \mathbbm{1}_{\partial G}\|_\infty \prec 1$ because: 
$$S_R \mathbbm{1}_{\partial G} = \int w(k)\widetilde{\pi}(k)  \mathbbm{1}_{\partial G}d\lambda(k) = \int w(k)\widetilde{P}_gd\lambda(k) \prec \int_{A_{2R}(2\alpha +3\tau)} e^{-2hR}\widetilde{P}_gd\lambda(k) \prec 1$$
Since $G$ is unimodular, 
$$S_R^* = \int w(g)\widetilde{\pi}(g^{-1}) d\lambda(g) = \int w(g^{-1})\pi(g) d\lambda(g) $$
(otherwise we would need right Haar measure). Using this formula one can now similarly show that $\|S_R^* \mathbbm{1}_{\partial G}\|_\infty \prec 1$.

Let $\phi, \psi \in L^4(\mu)$ and consider the measure $d\alpha = w(g)d\lambda(g)$. Using Cauchy-Shwartz and our estimates on $\|S_R \mathbbm{1}_{\partial G}\|_\infty, \|S_R^* \mathbbm{1}_{\partial G}\|_\infty$ we have:

\begin{align*}
&|\langle S_R \phi, \psi \rangle|^2 = \abs{\int_{G\times\partial G}\widetilde{P}_g(\xi)\phi(g^{-1}\xi)\overline{\psi(\xi)}d\alpha\times\mu(g,\xi)}^2
\\
\leq&
\abs{\int_{G\times\partial G}\widetilde{P}_g(\xi) |\phi(g^{-1}\xi)|^2 d\alpha\times\mu(g,\xi)}\abs{\int_{G\times\partial G}\widetilde{P}_g(\xi) |\psi(\xi)|^2 d\alpha\times\mu(g,\xi)} 
\\
=&
|\langle S_R |\phi|^2,\mathbbm{1}_{\partial G}\rangle||\langle S_R \mathbbm{1}_{\partial G}, |\psi|^2 \rangle|
=  |\langle |\phi|^2, S_R^* \mathbbm{1}_{\partial G}\rangle||\langle S_R \mathbbm{1}_{\partial G}, |\psi|^2 \rangle|
\\
\leq&
\|S_R^* \mathbbm{1}_{\partial G}\|_\infty \|S_R \mathbbm{1}_{\partial G}\|_\infty \|\phi^2\|_1 \|\psi^2\|_1 \prec \|\phi\|_2\|\psi\|_2
\end{align*}

The claim now follows by density of $L^4(\mu)$ in $L^2(\mu)$. 
\end{proof}

\begin{theorem}
Let $G$ be a unimodular, non-elementary, locally compact hyperbolic group. Let $d \in \mathcal{D}(G)$ and let $\mu$, $\pi$ be the corresponding measure and boundary representation. The kernel operators with non-negative kernels are in the positive cone of $\pi$.
\end{theorem}

\begin{proof}
Since the operators $S_R$ have uniformly bounded norms, in order to check convergence in the weak operator topology it is enough to check that $\langle S_R \phi,\psi \rangle \underset{R \to \infty}{\longrightarrow} \langle T_K \phi, \psi \rangle $ for $\phi, \psi$ in a dense subset of $L^2(\mu)$. When $K \geq 0$ we saw that this convergence holds for $\phi,\psi \in Lip(\partial G)$.
\end{proof}

We now prove one more lemma we will need later.

\begin{lemma} \label{projections in the positive cone}
    For any Borel set $E \subseteq \partial G$ the projection $P_E$ on to $L^2(E)$ is in the positive cone of $\pi$.
\end{lemma}

\begin{proof}
Consider the kernels $K_\rho(\xi,\eta) = \frac{1}{\mu(B_\rho(\xi))}\mathbbm{1}_{\{(\xi,\eta)| \xi \in E , d(\xi,\eta) < \rho\}}$ with corresponding kernel operators $T_\rho$ which are in the positive cone. We claim that $T_\rho \to P_E$ in the weak operator topology. As before we check convergence on Lipschitz functions first. Let $\phi, \psi \in Lip(\partial G)$ and $\xi \in \partial G$

\begin{align*}
    &\left |\int K_\rho(\xi,\eta)\phi(\eta) d\mu(\eta) - P_E\phi(\xi) \right|
    \\ = &
    \frac{\mathbbm{1}_E(\xi)}{\mu(B_\rho(\xi))} \left| \int_{B_\rho(\xi)} (\phi(\eta) -\phi(\xi))  d\mu(\eta)  \right| 
    \\ \leq &
    \frac{\mathbbm{1}_E(\xi)}{\mu(B_\rho(\xi))} \int_{B_\rho(\xi)} L(\phi)\rho d\mu
    \leq L(\phi)\rho
\end{align*}

Using this estimate we see that:
$$ |\langle (T_\rho - P_E)\phi,\psi \rangle| \leq \|\psi\|_\infty L(\phi) \rho \underset{\rho \to 0}{\longrightarrow} 0 $$

Now since $Lip(\partial G)$ is dense in $L^2 (\partial G)$ all that remains is to check that $\|T_\rho\|_{op}$ are bounded independently of $\rho$.
For any non negative kernel $K$ it follows from the Cauchy-Shwartz inequality on $(\partial G^2,\mu^2)$ that for any $\phi,\psi \in L^2(\mu)$

\begin{small}
$$\langle T_K\phi,\psi \rangle \leq \|\sqrt{K(\xi,\eta)}\phi(\eta)\|_2\|\sqrt{K(\xi,\eta)}\overline{\psi(\xi)}\|_2 \leq \left \|\int K(\xi,\eta) d\mu(\xi) \right \|_\infty^{\frac{1}{2}} \left \|\int K(\xi,\eta) d\mu(\eta) \right \|_\infty^{\frac{1}{2}} \|\phi\|_2\|\psi\|_2$$
\end{small}

But for $K = K_\rho$ it is easily seen that:

$$\int K_\rho(\xi,\eta) d\mu(\eta) , \int K_\rho(\xi,\eta) d\mu(\xi) \leq 1$$
So $\|T_\rho\| \leq 1$ as needed.

\end{proof}

\section{Irreducibility of Boundary Representations} \label{irreducibility}

We now prove irreducibility of the representation $\pi$.

\begin{lemma}
    Let $\pi$ be a representation of $G$, if the weak operator closure of $\pi(L^1(G))$ contains a projection on a cyclic vector then $\pi$ is irreducible.
\end{lemma}

\begin{proof}
We show that the commutant of $\pi$ is trivial, the result then follows by Schur's lemma. Let $P$ be a projection on to a cyclic vector $v$. If $P$ is in the weak operator closure of $\pi(L^1(G))$ then so is $\pi(g)P\pi(g)^{-1}$, which is the projection on $\pi(g)v$. So if $T$ is an operator commuting with $\pi(L^1(G))$ then $\pi(g)v$ is an eigenvector of $T$ for every $g \in G$. Since the vectors $\pi(g)v$ span a dense subspace we conclude that $T$ is a scalar operator.
\end{proof}

If G is a unimodular non-elementary locally compact hyperbolic group and $\pi$ is the boundary representation corresponding to some $d \in \mathcal{D}(G)$ then all the kernel operators are in the weak operator closure of $\pi(L^1(G))$. Let $\phi \in L^\infty(\mu)$ and consider the kernel operator $T_K$ corresponding to the kernel $K(\xi,\eta) = \phi(\xi)$. Note that $T_K(\mathbbm{1}_{\partial G}) = \phi$. Therefore since $L^\infty(\mu)$ is dense in $L^2(\mu)$, $\mathbbm{1}_{\partial G}$ is cyclic. Taking $\phi = \mathbbm{1}_{\partial G}$ we see that the projection on $\mathbbm{1}_{\partial G}$ is in the weak operator closure of $\pi(L^1(G))$. Thus we get by the lemma:

\begin{theorem}
    Let $G$ be a unimodular, non-elementary, hyperbolic locally compact group. Let $d \in \mathcal{D}(G)$ and let $\mu, \pi$ be the the corresponding Patterson-Sullivan measure and boundary representation then $\pi$ is irreducible.
\end{theorem}
Finishing the proof of theorem \ref{irreducibility main theorem}.

We remark that unimodularity of $G$ only entered the proof of irreducibility in two places: lemma \ref{cancellation lemma} and lemma \ref{uniformly bounded}, although the entire construction of the previous section would not work without these two lemmas.

\section{Rough Equivalence of Metrics} \label{rough equivalence of metrics}
This section is dedicated to the proof that two metrics $d_1,d_2 \in \mathcal{D}(G)$ are roughly similar if and only if the corresponding boundary representations are unitarily equivalent. The proof will use the fact that the action of $G$ on $\partial^2 G$ is erogodic which we will show later (theorem \ref{weak mixing}), independently of this section.

We fix throughout the section metrics $d_1,d_2 \in \mathcal{D}(G)$. We denote objects corresponding to $d_i$ by a subscript $i$. For example the boundary representations will be $\pi_i$ and the visual metrics will have parameters $\epsilon_i$ and be denoted $d_{\epsilon_i}$. From this point forward we add the assumption that $G$ is second countable to the standing assumption that $G$ is non elementary and unimodular.

\begin{lemma}
If the representations $\pi_1$ and $\pi_2$ are unitarily equivalent then there exists an almost everywhere defined measure class preserving isomorphism $F:\partial G \to \partial G$ conjugating $\pi_1$ to $\pi_2$.
\end{lemma}

\begin{proof}
Denote the ($G$ equivariant) unitary isomorphism $T:L^2(\mu_1) \to L^2(\mu_1)$. $T$ induces an isomorphism of the von Neumann algebras of bounded operators on these Hilbert spaces $\hat{T}: B(L^2(\mu_1)) \to B(L^2(\mu_2))$ by $\hat{T}(A) = TAT^{-1}$. We claim that $\hat{T}$ restricts to an isomorphism between the von Neumann algebras $L^\infty(\mu_i)$ considered as multiplication operators inside $B(L^2(\mu_i))$.

The projections in $L^\infty(\mu_i)$ can be characterized amongst the projections of $B(L^2(\mu_i))$ as those $P$ such that both $P$ and $I-P$ preserve the partial order on $L^2(\mu_i)$. Obviously all projections in $L^\infty(\mu_i)$ satisfy this. In the other direction if $P$ and $I-P$ both preserve the partial order on  $L^2(\mu_i)$ then $P \mathbbm{1}_{\partial G}$,$(I-P)\mathbbm{1}_{\partial G}$ are positive orthogonal functions, so they must have disjoint supports. Since $P\mathbbm{1}_{\partial G} +(I-P)\mathbbm{1}_{\partial G} = \mathbbm{1}_{\partial G}$ we conclude that $P\mathbbm{1}_{\partial G}$ is the indicator function of some set $E$. If $\phi \in L^\infty(\mu_i)$ is positive then $P\phi$ is supported on $E$ since $0 \leq P\phi < P \|\phi\|_\infty \mathbbm{1}_{\partial G} = \|\phi\|_\infty \mathbbm{1}_E$. Using linearity and density of $L^\infty(\mu_i)$ in $L^2(\mu_i)$ we deduce that $Im(P) \subseteq L^2(E)$. Similarly $Im(I-P) \subseteq L^2(E^c)$, so since $P, I-P$ are projections which sum up to $I$ we deduce that $Im(P) = L^2(E)$ as needed.

By lemma \ref{projections in the positive cone} we know that all the projections in $L^\infty(\mu_1)$ are in the positive cone of $\pi_1$. Using the $G$-equivariance we see that $\hat{T}$ takes the positive cone of $\pi_1$ in to the positive cone of $\pi_2$. Since elements of the positive cone of $\pi_2$ preserve the partial order on $L^2(\mu_2)$ we conclude that $\hat{T}$ takes projections in $L^\infty(\mu_1)$ to projections in $L^\infty(\mu_2)$. Finally because  the projections generate $L^\infty(\mu_1)$ we conclude that $\hat{T}(L^\infty(\mu_1)) \subseteq L^\infty(\mu_2)$. Applying the same line of reasoning to $\hat{T}^{-1}$ one gets the opposite inclusion showing that $\hat{T}$ restricts to a $G$-equivariant isomorphism of von Neumann algebras $\hat{T}:L^\infty(\mu_1) \to L^\infty(\mu_2)$

Since maps between commutative von Neumann algebras correspond bijectively to almost everywhere defined maps between the spectra we conclude that there exists a $G$-equivariant isomorphism $F: (\partial G, \mu_1) \to (\partial G, \mu_2)$.

\end{proof}

Our strategy now will be first to show that $F$ agrees almost everywhere with a continuous map and then to show that $F$ is the identity so that $\mu_1$ and $\mu_2$ are actually in the same measure class. In order to do this we will take a slight detour.

Denote by $\partial^2 G$ the space of distinct pairs of elements in $\partial G$. It will be slightly more convenient for us to consider only distinct pairs but since the measure $\mu$ is purely non-atomic this has no measure theoretic significance.

\begin{theorem} \label{invariant measure}
For any Patterson-Sullivan measure $\mu$, the measure space $(\partial^2 G,\mu^2)$ contains an invariant (infinite) measure $m$ satisfying:
$$ dm(\xi,\eta) \asymp e^{2h(\xi,\eta)}d\mu^2(\xi,\eta) \asymp 
 d_\epsilon(\xi,\eta)^{-2D}d\mu^2(\xi,\eta)$$
\end{theorem}

To prove this theorem we will need two lemmas.

\begin{lemma}
Let $G$ be a second countable locally compact group acting continuously on a locally compact space $X$. If $G$ has a dense subgroup $H$ preserving a (possibly infinite) Radon measure $m$ then $G$ preserves $m$.
\end{lemma}

\begin{proof}
Let $f: X \to \R$ be a continuous compactly supported function. Let $g \in G$ and $g_n \in H$ such that $g_n \to g$. If $U$ is some compact neighborhood of $g$ then the functions $f(g^{-1}x)$ are all supported on some compact subset of $X$. Since $g_n$ preserve $m$ we see by dominated convergence that:

$$\int f(x) dg_*m = \int f(g^{-1}x) dm = \lim_{n \to \infty}\int f(g_n^{-1}x) dm = \lim_{n \to \infty} \int f(x) d(g_{n})_{*}m = \int f(x) dm$$

Radon measures are defined by their values on continuous compactly supported functions so since $f$ is arbitrary we conclude that $g_*m = m$ as needed.
\end{proof}

\begin{lemma} \label{bounded RN}
    Let $G$ a locally compact second countable group acting continuously on a locally compact space $X$ and preserving the measure class of a Radon measure $\nu$ such that $\frac{dg_*\nu}{d\nu} \asymp 1$ uniformly in $g$. There exists an invariant measure $m \in [\nu]$ satisfying $\frac{dm}{d\nu} \asymp 1$.
\end{lemma}

\begin{proof}
    We start with the case of a countable group. Consider the function:
    $$ \alpha(x) = \sup_{h \in G} \frac{dh_*\nu}{d\nu}$$
    Since $G$ is countable  $\alpha $ is measurable and well defined almost everywhere.
    By the chain rule for Radon Nikodym derivatives:
    $$ \frac{d(gh)_*\nu}{d\nu}(x) = \frac{dh_*\nu}{d\nu}(g^{-1}x) \frac{dg_*\nu}{d\nu}(x) $$

    Taking the supremum over $h \in G$ we get that for every $g \in G$:
    $$\frac{dg_*\nu}{d\nu}(x) = \frac{\alpha(x)}{\alpha(g^{-1}x)}$$

    If we define a measure $m$ by $dm = \alpha(x)d\nu$ then for any $g \in G$:

    $$\frac{dg_*m}{dm}(x) = \frac{dg_*m}{dg_*\nu}(x)\frac{dg_*\nu}{d\nu}(x)\frac{d\nu}{dm}(x) = \frac{\alpha(g^{-1}x)}{\alpha(x)}\frac{dg_*\nu}{d\nu}(x) = 1$$
    So $m$ is invariant under $G$. By assumption $\alpha(x) \asymp 1$.

    If now we assume that $G$ is locally compact and second countable then then $G$ contains a dense countable subgroup $H$. By the countable case there is an $H$ invariant measure on $X$ satisfying the conditions of the theorem and by the previous lemma this measure is actually $G$ invariant as needed. 
\end{proof}

We are now ready to prove theorem \ref{invariant measure}.

\begin{proof}[proof of theorem \ref{invariant measure}]
    Consider the measure $d\nu = e^{2h(\xi,\eta)}d\mu^2$ on $\partial^2 G$. Since we are only considering distinct pairs in $\partial G$ this measure is a Radon measure (it would not be a Radon measure on $\partial G \times \partial G$). The theorem will now follow from lemma\ref{bounded RN} if we show $\frac{dg_*\nu}{d\nu} \asymp 1$ independently of $g$. Indeed:

    \begin{align*}
    ln \left(  \frac{dg_*\nu}{d\nu}(\xi,\eta) \right) 
    =&  ln \left(  \frac{dg_*\nu}{dg_*\mu^2}(\xi,\eta)\frac{dg_*\mu^2}{d\mu^2}(\xi,\eta)\frac{d\mu^2}{d\nu}(\xi,\eta) \right)
    \\
    \approx & 2h((g^{-1}\xi,g^{-1}\eta) + (g,\xi) + (g,\eta) - |g| - (\xi,\eta)) \approx 0
    \end{align*}
\end{proof}

Recall that given a metric space $(X,d)$ the cross ratio of four distinct points $x,y,z,w$ is:

$$[x,y;z,w] = \frac{d(x,z)d(y,w)}{d(x,w)d(y,z)}$$

We will show  in the following chapters (theorem \ref{weak mixing}) that the action of $G$ on $(\partial^2 G, \mu^2)$ is ergodic. The proof is independent rest of this chapter so we can use theorem \ref{weak mixing} and not run in to any circular logic. An ergodic measure class can not contain more then one invariant measure up to scalars. Therefore, denoting $F^2(\xi,\eta) = (F(\xi),F(\eta))$, we conclude that $F^2_*m_1 =\beta m_2$ for some constant $\beta > 0$. As a result:

$$\frac{dF^2_*\mu_1^2}{d\mu_2^2} (\xi,\eta) \asymp \frac{dF^2_*\mu_1^2}{dF^2_*m_1}(\xi,\eta)\frac{dm_2}{d\mu_2^2}(\xi,\eta) \asymp d_{\epsilon_1}(F^{-1}(\xi),F^{-1}(\eta))^{2D_1} d_{\epsilon_2}(\xi,\eta)^{-2D_2}$$

One can now calculate the cross ratio in $(\partial G, d_{\epsilon_1})$ directly and see that all the Radon-Nikodym derivatives cancel out giving:

$$ [F(\xi_1),F(\xi_2);F(\eta_1),F(\eta_2)]_2 \asymp [\xi_1,\xi_2;\eta_1,\eta_2]_1^{\frac{D_1}{D_2}}$$
which holds on a subset $E$ of full measure in $\partial G^4$.

\begin{lemma} \label{ae Holder equivalence}
    There exists a full measure subset of $ E'' \subseteq \partial G$ such that for $\xi,\eta \in E''$:
    $$d_{\epsilon_2}(F(\xi),F(\eta)) \prec d_{\epsilon_1}(\xi,\eta)^{\frac{D_1}{D_2}} $$
\end{lemma}

\begin{proof}
    Let $(\xi_2,\eta_2)$ be such that the fiber $\{(\xi_1,\eta_1) | (\xi_1,\xi_2,\eta_1,\eta_2) \in E\}$ has full measure in $\partial G \times \partial G$ and let $\rho > 0$. Thus the estimate on the cross ratios implies that for a full measure subset of $(\xi_1,\eta_1) \in B_\rho(\eta_2)^c \times B_\rho(\xi_2)^c$:

    $$d_{\epsilon_2}(F(\xi_1),F(\eta_1)) \prec_{\xi_2,\eta_2,\rho} d_{\epsilon_1}(\xi_1,\eta_1)^{\frac{D_1}{D_2}}$$

    By Fubini's theorem $(\xi_2,\eta_2)$ can be chosen from a full measure subset which in particular is dense in $\partial G \times \partial G$. By compactness, $\partial G \times \partial G$ can be covered by finitely many sets of the form $B_\rho(\eta_2)^c \times B_\rho(\xi_2)^c$ so that the above estimate actually holds uniformly over a full measure subset of $\partial G \times \partial G$. Since $F$ is an almost everywhere defined isomorphism there is a full measure subset $E'$ of $\partial G \times \partial G$ such that $d_{\epsilon_2}(F(\xi),F(\eta)) \prec d_{\epsilon_1}(\xi,\eta)^{\frac{D_1}{D_2}}$.

    Let $E''$ be the full measure subset of $\partial G$  such that the fibers over points in $E'$ have full measure. If $\xi,\eta \in E''$ and $\zeta$ is in the fiber over both $\xi$ and $\eta$ then:
    $$d_{\epsilon_2}(F(\xi),F(\eta)) \leq d_{\epsilon_2}(F(\xi),F(\zeta)) + d_{\epsilon_2}(F(\zeta),F(\eta)) \prec d_{\epsilon_1}(\xi,\zeta)^{\frac{D_1}{D_2}} +d_{\epsilon_1}(\zeta,\eta)^{\frac{D_1}{D_2}}$$
    Since $\zeta$ can be taken from a full measure subset we can let it converge to $\eta$ and we get $d_{\epsilon_2}(F(\xi),F(\eta)) \prec d_{\epsilon_1}(\xi,\eta)^{\frac{D_1}{D_2}}$ for all $\xi,\eta \in E''$ as needed.
 \end{proof}

As a result we see that $F$ is uniformly continuous on $E''$ and thus agrees almost everywhere with a uniformly continuous function $H: \partial G \to \partial G$. Since $F$ is almost everywhere $G$-equivariant we conclude that $H$ is $G$-equivariant everywhere so by lemma \ref{no automorphisems} $H$ is trivial. Finally using the symmetry between $d_1$ and $d_2$, for all $\xi, \eta \in \partial G \times \partial G$:

$$d_{\epsilon_2}(\xi,\eta) \asymp d_{\epsilon_1}(\xi,\eta)^{\frac{D_1}{D_2}}$$

Now we can prove the main theorem of this section:

\begin{theorem}
Let $G$ be a second countable non elementary unimodular locally compact hyperbolic group and let $d_1,d_2 \in \mathcal{D}(G)$. The following are equivalent:
\begin{enumerate}
    \item The metrics $d_1$ and $d_2$ are roughly similar.
    \item The Patterson-Sullivan measures $\mu_1$ and $\mu_2$ are in the same measure class.
    \item The boundary representations $\pi_1$ and $\pi_2$ are unitarily equivalent.
\end{enumerate}
\end {theorem}

\begin{proof}
If $d_2 \approx Ld_1$ are roughly similar then for a given small enough visual parameters $\epsilon$ (chosen for simplicity to be the same for both metrics), the visual metrics $d^1_{L\epsilon}$ and $d^2_\epsilon$ are bi-Lipschitz and thus produce Hausdorff measures in the same measure class so by lemma \ref{PS are Hausdorff} $\mu_1$ and $\mu_2$ are in the same measure class. As explained in the preliminaries this implies that $\pi_1$ and $\pi_2$ are equivalent.

For the implication $(3) \implies (1)$, since $d_{\epsilon_2}(\xi,\eta) \asymp d_{\epsilon_1}(\xi,\eta)^{\frac{D_1}{D_2}}$ the corresponding $D_i$ -dimensional Hausdorff measures are in the same measure class with Radon-Nikodym derivatives bounded away from $0$ and $\infty$. By lemma \ref{PS are Hausdorff} the same is true for $\mu_1$ and $\mu_2$. Therefore for any $g \in G$:

$$e^{h_1(2(g,\xi)_1 -|g|_1)} \asymp e^{h_2(2(g,\xi)_2 -|g|_2)}$$

By lemma \ref{nonempty shadow}, after taking suprema over $\xi$ we get that $e^{h_1|g|_1} \asymp e^{h_2|g|_2}$ so:

$$|g|_1 \approx \frac{h_2}{h_1}|g|_2$$

as needed.
\end{proof}

We have now finished the proof of theorem \ref{classification main theorem}, assuming theorem \ref{weak mixing} holds.

\section{Type I Hyperbolic Groups} \label{type I hyperbolic groups}
One of the original motivations for generalizing Garncarek's work \cite{G14} to the locally compact setting was that such a generalization would strengthen the results of Caprace, Klantar and  Monod in \cite{CKM21}. This is pointed out explicitly in \cite[Remark~5.6]{CKM21}. We describe the strengthened results here.

\begin{definition}
    A locally compact group $G$ is of type I if any two weakly equivalent (see \cite[Appendix~F.1]{BDV08}) irreducible unitary representations are unitarily equivalent.
\end{definition}

If $G$ is a second countable totally disconnected unimodular hyperbolic locally compact group then it follows from theorem \ref{irreducibility main theorem} that for any compact open subgroup $U$ and any Cayley-Abels graph on $G/U$ the Patterson-Sullivan representation corresponding to the induced (pseudo)metric on $G$ is irreducible. Similarly by theorem \ref{classification main theorem} the identity on $G/U$ is a rough similarity between two Cayley-Abels graphs if and only if the corresponding Patterson-Sullivan representations are unitarily equivalent. These are exactly conditions (G1) and (G2) mentioned in \cite[Remark~5.6]{CKM21}, thus it follows from \cite[Remark~5.6]{CKM21} that \cite[Theorem~B]{CKM21},\cite[Theorem~K]{CKM21} hold in greater generality. Explicitly:

\begin{theorem}
     Let $G$ be a second countable unimodular hyperbolic locally compact group. If $G$ is of type I then $G$ admits a cocompact amenable subgroup.
\end{theorem}

\begin{theorem}
Let $G$ be a non-amenable second countable unimodular hyperbolic locally compact group. Let $\mu$ be a quasi invariant measure on $\partial G$, $\pi$ the corresponding Koopman representation and $C^*_\pi(G)$ the image of the group $C^*$-algebra $C^*(G)$ under $\pi$. The following are equivalent:

\begin{itemize}
    \item $C^*_\pi(G)$ contains a non-zero CCR closed two sided ideal.
    \item $C^*_\pi(G)$ is GCR.
    \item $C^*_\pi(G)$ is CCR.
    \item $C^*_\pi(G)$ consists entirely of compact operators.
    \item G has a cocompact amenable subgroup.
\end{itemize}
\end{theorem}
Note that the last item of the proposition is independent of $\mu$.

It is an interesting question whether the unimodularity assumption in these theorems can be replaced by the weaker assumption of non-amenability. This would follow as above if one would show the conclusions of theorems \ref{irreducibility main theorem} and \ref{classification main theorem} hold for non amenable hyperbolic groups.

\section{The Geodesic Flow} \label{the geodesic flow}
We retain with the assumptions from the last sections that $G$ is second countable locally compact hyperbolic and unimodular.
In the special case when $G$ is a discrete group acting freely properly and co-compactly on a negatively curved simply connected manifold $M$ one can associate to this action the geodesic flow on the unit tangent bundle of the quotient $T^1N = G\backslash T^1M$. In this restricted setting $T^1M$ is canonically identified with the space of pointed geodesics in $M$. By looking at the endpoints of each geodesic we see that this space is a line bundle over $\partial^2 G$. Measure theoretically we can identify this bundle with $\partial^2 G \times \R$ and the action of $G$ on 
on this space is given by a certain cocycle $\tau: G \times \partial^2 G \to \R$. This point of view can be generalized to our setting and will allow us to define a geodesic flow for general $(G,d)$. In the case where $G$ is a discrete group this is done in \cite[Appendix~A.1]{G14} and \cite[Section~3]{BF17}. We will encounter some added measure theoretic difficulties. The two major complications are the following:

\begin{itemize}
    \item If $G$ is uncountable not every null set is contained in a $G$-invariant null set so one must take care to ignore only invariant null sets.
    
    \item If $G$ is uncountable a fundamental domain for a non-singular action of $G$ will typically be a null set (which is of course not $G$-invariant).
\end{itemize}

Define $\sigma, \rho: G \times \partial G \to \R$ and $\tau: G \times \partial^2G \to \R$ by:

$$\sigma(g,\xi) = 2(g^{-1},\xi) - |g|$$
$$\rho(g,\xi) = \frac{1}{h}ln\frac{dg^{-1} _*\mu}{d\mu}(\xi)$$
$$\tau(g,\xi,\eta) = \frac{\rho(g,\eta)-\rho(g,\xi)}{2} \approx (g^{-1},\eta) - (g^{-1},\xi) $$ 

By the chain rule for Radon-Nikodym derivatives $\rho$ and $\tau$ are cocycles, i.e. they satisfy the cocycle equation:
$$c(gh,x) = c(g,hx) + c(h,x)$$
By \cite[Appendix~B.9]{ZIM84} we can assume that $\rho$ and $\tau$ are strict cocycles, i.e. that the cocycle equation is satisfied everywhere (as opposed to almost everywhere).
$\sigma$ is a strict almost cocycle in the sense that for all every $g,h$ and $\xi$:

$$\sigma(gh,\xi) \approx \sigma(g,h\xi) + \sigma(h,\xi)$$
uniformly in $g,h,\xi$.

Finally, by quasi-conformality of $\mu$ it follows that for every $g \in G$, for almost every $\xi \in \partial G$:

\begin{equation} \label{close to cocycle}
    \rho(g,\xi) \approx \sigma(g,\xi)
\end{equation}

By appendix \ref{almost cocycle lemma}, after replacing $\rho$ by a strictly cohomologous cocycle we can assume that estimate \ref{close to cocycle} holds on a $G$-invariant full measure subset $A \subseteq \partial G$. Abusing notation we call will denote this new cocycle again by $\rho$. We now get also that for all $g \in G$ and $(\xi,\eta)$ in a full measure $G$-invariant set $S$ (for example $ S = A^2-\Delta A$) we have:
$$\tau(g,\xi,\eta) \approx \frac{\sigma(g,\eta) - \sigma(g,\xi)}{2}$$

We will need the following lemma in the next section:

\begin{lemma} \label{bound on tau minus sigma}
    For every $M>0$ there exists a constant $C_M$ such that if $(\xi,\eta) \in S$ such that $(\xi,\eta), (g\xi,g\eta) < M$ then $|\tau(g,\xi,\eta) - \sigma(g,\eta)|<C_M $.
\end{lemma}

\begin{proof}
    Since $\sigma(g,\eta) \approx \rho(g,\eta)$ we have that $\tau(g,\xi,\eta) - \sigma(g,\eta) \approx -\frac{\rho(g,\eta)+\rho(g,\xi)}{2}$.  Using the chain rule for Radon Nikodym derivatives:

    $$1 = \frac{dg^{-1}_*m}{dm}(\xi,\eta) = \frac{dm}{d\mu^2}(g\xi,g\eta)\frac{dg^{-1}_*\mu^2}{d\mu^2}(\xi,\eta)\frac{d\mu^2}{dm}(\xi,\eta) $$
    Taking the logarithm of both sides and recalling that $\frac{dm}{d\mu^2}(\xi,\eta) \asymp e^{2h(\xi,\eta)}$ we see that $|\frac{\rho(g,\eta)+\rho(g,\xi)}{2}| \approx |(\xi,\eta) - (g\xi,g\eta)| \leq 2M $  as needed.
\end{proof}

Fix an ergodic p.m.p action of $G$ on a standard probability space $(\Omega,\omega)$, we introduce this space to increase the generality of the next section but it will have very little effect on the discussion and for our purposes one can take $\Omega$ to be trivial. Recall the invariant measure $m$ constructed on $\partial^2 G$. Denote the Lebesgue measure on $\R$ by $\ell$. Using the cocycle $\tau$ we can define an (infinite) measure preserving action of $G$ on the space $(S \times \R \times \Omega, m \times  \ell \times \omega)$ by:

$$g(\xi,\eta,t,w) = (g\xi,g\eta,t+\tau(g,\xi,\eta),gw)$$
This action commutes with the $\R$-flow defined by:

$$\Phi^s (\xi,\eta,t) = (\xi,\eta,t + s)$$
Denote:

$$D_{\theta,k} = \{(\xi,\eta, t, w) \in S \times (-k,k) \times \Omega | d_\epsilon (\xi,\eta) > \theta\}$$
We call sets contained in some $D_{\theta,k}$ bounded. Being bounded is equivalent to having precompact projection in $\partial^2 G \times \R$. Bonded sets have finite $m \times \ell \times \omega$ measure.

Similarly to the proofs in of \cite[Appendix~A.2,A.3]{G14} one proves the following two lemmas:

\begin{lemma} \label{proper action}
For any $\theta,k > 0$, $\{g\in G | gD_{\theta,k} \cap D_{\theta,k} \neq \phi\}$ is bounded in the metric $d$ and precompact in $G$.
\end{lemma}

\begin{proof}
    Suppose that  $gD_{\theta,k} \cap D_{\theta,k} \neq \phi$. There exists $(\xi,\eta,t,w) \in D_{\theta,k}$ such that $(g\xi,g\eta,t + \tau(g,\xi,\eta),gw) \in D_{\theta,k}$. Therefore, $(\xi,\eta),(g\xi,g\eta) \approx 0$ and $|(g^{-1},\eta) -(g^{-1},\xi)| \approx |\tau(g,\xi,\eta)| \approx 0$. Now $(\xi,\eta) \gtrsim \min\{(\xi,g^{-1}),(g^{-1},\eta)\}$ but $(\xi,\eta) \approx 0$ and $(g^{-1},\eta) \approx (g^{-1},\xi)$ so $(\xi,g^{-1})\approx(g^{-1},\eta) \approx 0$.
    Using lemma \ref{basic contraction lemma} we see that:
    $$(g\xi,g\eta) \gtrsim \min\{(g\xi,g),(g,\eta)\} \approx |g| - \max\{(g^{-1},\xi),((g^{-1},\eta)\} \approx |g|$$
    But $(g\xi,g\eta) \approx 0$ so $|g| \approx 0$ as needed.
\end{proof}

\begin{lemma} \label{cocompact action}
For large enough $\theta$ and  $k$, $D_{\theta,k}$ intersects every $G$ orbit in $S \times \R \times \Omega$.
\end{lemma}

\begin{proof}
    Let $(\xi,\eta,t,w) \in S \times \R\times \Omega$ and let $\gamma :(-\infty,\infty) \to G$ be a rough geodesic with $\gamma(\infty) = \eta, \gamma(-\infty) =\xi$. For any $r \in \R$ we have that $(\xi,\eta)_{\gamma(r)} \approx 0$ and thus $(\gamma(r)^{-1}\xi,\gamma(t)^{-1}\eta) \approx 0$. 
    In addition:
    \begin{eqnarray*}
    \sigma(\gamma(r)^{-1},\eta) & = & 2(\gamma(r),\eta) - |\gamma(r)| \approx \liminf_{s \to \infty} |\gamma(r)|-(s-r) \\
    & = & r + \liminf_{s\to\infty}(|\gamma(s)| - s) 
    = r + \sigma(\gamma(0)^{-1},\eta) 
    \end{eqnarray*}
    and similarly $\sigma(\gamma(r)^{-1},\xi) \approx - r + \sigma(\gamma(0)^{-1},\xi)$ so:
    $$\tau(\gamma(r)^{-1},\xi,\eta) \approx r + \tau(\gamma(0)^{-1},\xi,\eta)$$

    Putting everything together we see that for $r = - \tau(\gamma(0)^{-1},\xi,\eta) -t$: $(\gamma(r)^{-1}\xi,\gamma(r)^{-1}\eta) \approx 0$ and $\tau(\gamma(r)^{-1},\xi,\eta) \approx 0$. Therefore $\gamma(r)^{-1}(\xi,\eta,t,w)$ is in some $D_{\theta,k}$ with $\theta$ and $k$ independent of $(\xi,\eta,t,w)$. 
\end{proof}

\begin{lemma} \label{bounded times bounded is bounded}
If $U \subset G$ is bounded then for any $\theta$ and $k$, $UD_{\theta,k}$ is bounded.
\end{lemma}

\begin{proof}
Given $g \in G$ and $(\xi,\eta) \in \partial^2 G$ we have that $|\tau(g,\xi,\eta)| \approx |(g^{-1},\xi)-(g^{-1},\eta)| \lesssim 2|g|$ and $|(g\xi,g\eta) - (\xi,\eta)| \approx |(\xi,\eta)_{g^{-1}} - (\xi,\eta)| \lesssim |g|$. Together with the boundedness of $U$ this shows that for any $\theta$ and $k$ there exist $\theta'$ and $k'$ such that $U D_{\theta,k} \subseteq D_{\theta',k'}$.
\end{proof}

Intuitively one should think of lemma \ref{proper action} and lemma \ref{cocompact action} as saying that the $G$ action on $S \times \R \times \Omega$ is proper and co-compact, although formally this is a measurable action and not a continuous topological action. In the same spirit we have the following theorem:

\begin{theorem}
The $G$ action on $S \times \R \times \Omega$ admits a bounded Borel cross section, i.e. a bounded Borel subset $\hat{X}$ intersecting each $G$ orbit exactly once.
\end{theorem}

\begin{proof}
By \cite[Proposition~5.10]{CCMT15} $G$ has a maximal normal compact subgroup $W$ and $G/W$ is either a virtually connected rank one simple adjoint Lie group or $G/W$ is totally disconnected.

In case $G/W$ is connected the action of $G$ on $\partial G$ is transitive so there is a unique invariant measure class on $\partial G$ and $\mu$ is in this measure class so the cocycle $\tau$ is independent of $d$ up to cohomology and the actions on $S\times\R\times\Omega$ given by different $d$ are conjugate. For the specific (psudo)metric pulled back from the corresponding symmetric space this action is simply the action on the unit tangent bundle of the symmetric space which is known to be transitive so any point is a cross section.

In case $G/W$ is totally disconnected, by van Dantzig's theorem $G$ admits a compact open subgroup $U$. 
Since $U$ is compact it is bounded in the metric $d$. 

By \cite[Corrolary~2.1.21]{ZIM84} and \cite[Appendix A.7]{ZIM84} the $U$ action on $S \times \R \times \Omega$ admits a Borel cross section $\Delta$.
By lemma \ref{cocompact action} there exists some $D_{\theta',k'}$ intersecting every $G$ orbit and by lemma \ref{bounded times bounded is bounded} there exist $\theta,k$ such that $UD_{\theta',k'} \subseteq D_{\theta,k}$. Since every point has a unique point of $\Delta$ in its $U$ orbit every $G$ orbit intersects $D_{\theta,k} \cap \Delta$ non trivially. By lemma \ref{proper action} there exists a compact set $F \subset G$ such that if $g \notin F$, $gD_{\theta,k} \cap D_{\theta,k} = \phi$. We will importantly use the fact that this intersection is empty and not only null, this is the reason we are restricting ourselves from $\partial^2 G$ to $S$. Enlarging $F$ we can assume without loss of generality that $F = \bigcup_i Ug_i$ is a finite union of right cosets of U. If $x \in S\times \R \times \Omega$ and $g\in G$ such that $x,gx \in D_{\theta,k} \cap \Delta$ then $g \in Ug_i$ for some $i$, but $|Ug_i x \cap \Delta| = 1$. Therefore the intersection of each $G$ orbit with $D_{\theta,k} \cap \Delta$ is finite and non empty.

To sum up, every $G$ orbit intersects $D_{\theta,k} \cap \Delta$ and the restriction of the $G$ orbit equivalence relation to $D_{\theta,k} \cap \Delta$ has finite  equivalence classes. Using the finite Borel selection theorem we deduce that there is a Borel subset $\hat{X} \subseteq D_{\theta,k} \cap \Delta$ intersecting every $G$ orbit exactly once as needed. This set is bounded since it is contained in $D_{\theta,k}$.
\end{proof}

We now get a Borel isomorphism between the the space $X = (S\times \R \times \Omega) // G$ of $G$-ergodic components and $\hat{X}$ by restricting the projection $p: S\times \R \times \Omega \to X$ to $\hat{X}$. This identifies the orbit space of $G$ with the space of ergodic components. Since the flow $\Phi$ on $(S\times \R \times \Omega)$ commutes with $G$ it descends to a flow $\phi$ on $X$ such that $p \circ \Phi^s = \phi^s \circ p$.

Since $G$ is unimodular and for any compact $U \subseteq G$, $U\hat{X}$ is bounded, by theorems \ref{measure construction} and \ref{invariant measure on quotient} there exists a unique measure $\nu$ on $X$ in the measure class of $p_*(m \times \ell \times \omega)$ which is invariant under the flow $\phi$ and such that for any $f \in L^1(S\times \R \times \Omega, m\times \ell \times \omega)$:

\begin{equation} \label{disintegration}
\int f(z) dm\times \ell \times \omega(z) = \int\int f(gz) d\lambda(g)d\nu(p(z))
\end{equation}
(Note that the inner integral on the right hand side depends only on $p(z)$ and not on $z$ because $G$ is unimodular.)

Since for any compact $U \subset G$, $U\hat{X}$ is bounded, it follows form equation \ref{disintegration} that $\nu$ is a finite measure and after re-scaling $m$ we can assume without loss of generality that $\nu$ is a probability measure.
If we take $\Omega$ to be trivial then we call the associated flow $(X,\nu,\phi)$ the \textbf{geodesic flow} of $(G,d)$.

\section{Double Ergodicity} \label{double ergodicity}
We keep all the notation and assumptions from the previous section.
In this section we will prove that the action of $G$ on $\partial^2 G$ is ergodic and even weak mixing, meaning that for any ergodic p.m.p action $G \acts (\Omega,\omega)$ the diagonal action on $(\partial^2 G \times \Omega, m \times \omega)$ is ergodic. We will use a classic argument of E. Hopf and adapt the proof in \cite[Section~4]{BF17} to the locally compact setting. The author would like to thank Uri Bader and Alex Furman for sending him a forthcoming unpublished version of \cite{BF17} on which our arguments are very closely based.

The space $(S\times \R \times \Omega, m\times \ell \times \omega)$ is equipped with the commuting measure preserving actions of $G$ and the $\R$-flow $\Phi$ discussed in the previous section. The space of $\R$-ergodic components is $(S\times \Omega,[m \times \omega])$ and in its canonical measure class it supports the $G$-invariant infinite measure $m \times \omega$. The space $(X,[\nu])$ of $G$-ergodic components has the natural $\R$-flow $\phi$ and contains the finite measure $\nu$ in its canonical measure class. $\phi$ is $\R$-invariant and satisfies equation \ref{disintegration}. After re-scaling $m$ we can further assume that $\nu$ is a probability measure. The $\R$-ergodic components of $(X,\nu)$ , the $G$-ergodic components of $(S \times \Omega, m \times \omega)$ and the $G \times \R$-ergodic components of $(S\times \R \times \Omega, m\times \ell \times \omega)$ are all given by the same space $(Y,[\beta])$. In the measure class $[\beta]$ we choose the measure $\beta$ which is the image of $\nu$. We thus obtain the following commutative diagram:

\begin{equation} \label{Hopf argument maps}
\begin{tikzcd}
                                                                   & {(S\times \mathbb{R} \times \Omega, m\times\ell\times\omega)} \arrow[ld, "p"'] \arrow[rd, "q"] &                                                   \\
{(X,\nu)} \arrow[rd, "u"] \arrow[ru, "\bar{p}", dotted, bend left] &                                                                                                & {(S\times\Omega, m\times\omega)} \arrow[ld, "v"'] \\
                                                                   & {(Y,\beta)}                                                                                    &                                                  
\end{tikzcd}
\end{equation}

The maps $p$ and $v$ are the $G$-ergodic components maps, the maps $q$ and $u$ are the $\R$-ergodic components maps and the map $u\circ p = v \circ q$ is the $G \times \R$-ergodic components map. The maps $p$ and $q$ are $\R$-equivariant and $G$-equivariant respectively. 
Ergodicity of $G$ on $(S \times \Omega, m \times \omega)$, ergodicity of $\R$ on $(X,\mu)$ and ergodicity of $G\times\R$ on  $(S\times \R \times \Omega, m\times \ell \times \omega)$ are all equivalent to $\beta$ being supported on a single point.

In the previous section we constructed a bounded cross section $\hat{X}$ for the $G$ action on $(S \times \R \times \Omega, m\times \ell \times \omega)$. The map $\bar{p}: X \to S\times\R\times\Omega $ describes the corresponding section of $p$. Unlike the case of discrete $G$ the measure $\bar{p}_*\nu$ might be singular to $m\times \ell \times \omega$. Given $x \in X$ we denote

$$\bar{p}(x) = (x_-,x_+,t_x,\omega_x)$$

The maps $p,q,u$ and $v$ can be used to push forward finite signed measure which are absolutely continuous with respect to the given measure classes. By the Radon-Nikodym theorem the space of finite signed measures absolutely continuous with respect to a given measure is identified with the space of integrable functions. Under these isomorphisms we obtain the operators $P,Q,U,V$ in the following commutative diagram:

\begin{equation}\label{Hopf argument operators}
\begin{tikzcd}
                             & {L^1(S\times \mathbb{R} \times \Omega, m\times\ell\times\omega)} \arrow[ld, "P"'] \arrow[rd, "Q"] &                                                                                                    \\
{L^1(X,\nu)} \arrow[rd, "U"] &                                                                                                   & {L^1(S\times\Omega, m\times\omega)} \arrow[ld, "V"'] \arrow[lu, "R_{\theta}"', dotted, bend right] \\
                             & {L^1(Y,\beta)}                                                                                    &                                                                                                   
\end{tikzcd}
\end{equation}

We will now give explicit descriptions of $Q,P$ and $U$. The operator $Q$ is given by integration over the $\R$ coordinate, For any $f \in L^1(S\times \R \times \Omega, m\times \ell \times \omega)$:

$$Q(f)(\xi,\eta,w) = \int f(\xi,\eta,t,w) dt$$
By equation \ref{disintegration}, $P$ is given by integration over $G$ orbits, or equivalently over fibers of $p$. Given $f \in L^1(S\times \R \times \Omega, m\times \ell \times \omega)$:

$$P(f)(p(\xi,\eta,t,w)) = \int f(g(\xi,\eta,t,w)) d\lambda(g)$$
Note that since $G$ is unimodular the right hand side indeed depends only on $p(\xi,\eta,t,w)$ (this fact is critical in the construction of the measure $\nu$).

Using Birkhoff's ergodic theorem we see that $U$ is given by averaging over $\R$ orbits in $X$. For $f \in L^1(X,\nu)$, for almost every $x \in X$, for any $a\leq b \in \R$:

$$U(f)(u(x)) = \frac{1}{T}\int_a^{b+T} f(\phi^{-t}x)dt$$

Our next goal will be to obtain an explicit description of $V$. For this we need to explain the doted arrow $\R_\theta$ in diagram \ref{Hopf argument operators}. If $\theta \in L^1(\R)$ is a positive kernel, i.e $\theta: \R \to [0,\infty)$ and:

$$\int \theta(t) dt = 1$$
We define $R_\theta: L^1(S\times \Omega, m \times \omega) \to L^1(S\times \R \times \Omega, m \times \ell \times \omega)$ by:

$$\R_\theta(f)(\xi,\eta,t,w) = \theta(t)f(\xi,\eta,w)$$
Clearly $Q \circ R_\theta = Id$. Since $V \circ Q = U \circ P$ we can precompose with $R_\theta$ to get that $V = U \circ P \circ R_\theta$. Given $f \in L^1(S\times \Omega, m \times \omega)$ denote:

\begin{equation} \label{description of bar sub theta}
    \bar{f}_\theta(x) = P\circ R_\theta (f) = \int R_\theta f(g\bar{p}(x))d\lambda(g) = \int \theta(t_x +\tau(g,x_-,x_+))f(gx_-,gx_+,gw_x) d\lambda(g)
\end{equation}
We also denote:
$$\bar{f}_{[0,1]} = \bar{f}_{\mathbbm{1}_{[0,1]}}$$

Using $R_\theta$ we obtain the following description of $V$. For any $a \leq b$, for almost every $y \in Y$, if $y = u(x)$:

\begin{equation} \label{first description of V}
   V(f)(y) = \frac{1}{T}\int_a^{b+T} \bar{f}_{[0,1]}(\phi^{-t}x)dt 
\end{equation}
While this formula is explicit it describes the function $V(f)$ in terms of the parameter $x \in X$. We would like to obtain a description of $V$ in terms of a points $(\xi,\eta,w) \in S\times \Omega$. For this purpose we introduce the following averaging operators. For any $a \leq b$ and $f \in L^1(S\times \Omega,m \times \omega)$ define $I_a^b(f),J_a^b(f): S \times \Omega \to [0,\infty]$ by:

$$I_a^b(f)(\xi,\eta,w) = \frac{1}{b-a}\int_{\{g \in G| \tau(g,\xi,\eta) \in [a,b]\}} f(g\xi,g\eta,gw) d\lambda(g)$$

$$J_a^b(f)(\xi,\eta,w) = \frac{1}{b-a}\int_{\{g \in G| \sigma(g,\eta) \in [a,b]\}} f(g\xi,g\eta,gw) d\lambda(g)$$
We prove the following ergodic theorem describing $V$:

\begin{theorem} \label{ergodic theorem for V}
    For any $f \in L^1(S\times \Omega,m \times \omega)$, for almost every $(\xi,\eta,w)$ the values $I_a^b(f)(\xi,\eta,w)$ and $J_a^b(f)(\xi,\eta,w)$ are finite for any interval $[a,b]$ and if $y = v(\xi,\eta,w)$ then:
    $$V(f)(y) = \lim_{T \to \infty} I_a^{b+T}(f)(\xi,\eta,w)$$
    If in addition $f$ is supported on a bounded set then:
    $$V(f)(y) = \lim_{T \to \infty} J_a^{b+T}(f)(\xi,\eta,w)$$
\end{theorem}

The proof will require two lemmas.

\begin{lemma}\label{convolution lemma}
    For any two positive kernels $\theta_1, \theta_2$ and $f \in L^1(S\times \Omega,m \times \omega)$:
    $$\bar{f}_{\theta_2*\theta_1} = \int_{-\infty}^{\infty} \theta_2(s)\bar{f}_{\theta_1}\circ\phi^{-s} ds$$
\end{lemma}

\begin{proof}
    Notice that:
    $$R_{\theta_2 * \theta_1}f(\xi,\eta,t,w) = \int_{-\infty}^{\infty} \theta_2(s)\theta_1(t-s)f(\xi,\eta,w)ds = \int_{-\infty}^{\infty} \theta_2(s)R_{\theta_1}f\circ\Phi^{-s}(\xi,\eta,t,w)ds$$
    We finish by applying $P$ to both sides.
\end{proof}

\begin{lemma} \label{avrages for points in the cross section}
    There exists a constant $c\geq0$ such that for every $m\times\omega$ integrable function $f:S\times\Omega \to [0,\infty)$, for every interval $[a,b]$ with $a+c+1 \leq b$ and for every $x \in X$:

    $$\frac{1}{b-a} \int_{a+c}^{b-c-1} \bar{f}_{[0,1]} \circ \phi^{-t}(x)dt \leq I_a^b(f)(x_-,x_+,w_x) \leq \frac{1}{b-a} \int_{a-c-1}^{b+c} \bar{f}_{[0,1]} \circ \phi^{-t}(x)dt$$
\end{lemma}

\begin{proof}
    Denote $\theta_a^b = \mathbbm{1}_{[a,b]}*\mathbbm{1}_{[0,1]}$ and notice that if $a + 1\leq b$ then $\theta_a^b$ interpolates linearly between the values $0$ on $(-\infty,a] \cup [b+1,\infty)$ and $1$ on $[a+1,b]$. In particular $ \mathbbm{1}_{[a+1,b]} \leq \theta_a^b \leq \mathbbm{1}_{[a,b+1]}$.
    Recall that the cross section $\hat{X}$ is bounded, therefore there exists a constant $c$ such that for any $x \in X$, if $\bar{p}(x) = (x_-,x_+,t_x,w_x)$ then $|t_x| < c$.
    Let $x \in X$ and $a,b \in \R$ such that $a + c + 1 \leq b$, using lemma \ref{convolution lemma} and equation \ref{description of bar sub theta} we see that:

    \begin{align*}
        &\int_a^b \bar{f}_{[0,1]}\circ\phi^{-t}(x)dt 
        \\=&\int_{-\infty}^\infty \mathbbm{1}_{[a,b]}(t)\bar{f}_{[0,1]}\circ\phi^{-t}(x)dt = \bar{f}_{\theta_a^b}(x) 
        \\=& \int_G \theta_a^b(t_x +\tau(g,x_-,x_+))f(gx_-,gx_+,gw_x) d\lambda(g)
    \end{align*}
Since $|t_x| < c$ and $\mathbbm{1}_{[a+1,b]} \leq \theta_a^b \leq \mathbbm{1}_{[a,b+1]}$ we conclude that:

\begin{small}
$$ \frac{b-a-2c-1}{b-a}I_{a+c+1}^{b-c}f(x_-,x_+,w_x) \leq 
\frac{1}{b-a}\int_a^b \bar{f}_{[0,1]}\circ\phi^{-t}(x)dt \leq
\frac{b-a+2c+1}{b-a}I_{a-c}^{b+c+1}f(x_-,x_+,w_x)$$
\end{small}
Rewriting these inequalities we get the desired result.
\end{proof}

\begin{proof}[proof of theorem \ref{ergodic theorem for V}]

We first prove the claim for the operator $I$. Let $f:S\times\Omega \to [0,\infty)$ be $m\times\omega$ integrable. By Birkhoff's ergodic theorem there exists a full measure subset $X_0 \subseteq X$ such that for $x \in X_0$ and any $a < b$:
$$V(f)(u(x)) = \lim_{T\to \infty}\frac{1}{T}\int_a^{b+T} \bar{f}_{[0,1]}(\phi^{-t}x)dt$$
$X_0$ is an $\R$-invariant subset so $p^{-1}(X_0)$ is a $G\times \R$-invariant subset in $S\times\R\times \Omega$ and $A_f = q(p^{-1}(X_0))$ is a $G$-invariant Borel subset of $S\times\Omega$. We will show that the desired convergence holds on $A_f$. 

Let $(\xi,\eta,w) \in A_f$, consider $x = p(\xi,\eta,0,w) \in X_0$. There exists $h \in G$ such that $(\xi,\eta,0,w) = h(x_-,x_+,t_x,w_ x) = h\bar{p}(x)$. Denoting $s = \tau(h,x_-,x_+)$ we have that $\tau(g,\xi,\eta) =  \tau(gh,x_-,x_+) - s$ so $\tau(g,\xi,\eta) \in [a,b]$ if and only if $\tau(gh,x_-,x_+) \in [a+s,b+s]$. Therefore:

\begin{align*}
&I_a^b(f)(\xi,\eta,w) 
=\frac{1}{b-a} \int_{\{g\in G | \tau(g,\xi,\eta) \in [a,b]\}} f(ghx_-,ghx_+,ghw_x) d\lambda(g) 
\\ =& \frac{1}{b-a}\int_{\{g\in G | \tau(gh,x_-,x_+) \in [a+s,b+s]\}} f(ghx_-,ghx_+,ghw_x) d\lambda(g) 
= I_{a+s}^{b+s}(f)(x_-,x_+,w_x) 
\end{align*}

Note that the last equality uses the assumption that $G$ is unimodular. By lemma \ref{avrages for points in the cross section}, for any $a,b$ such that $a+c+1 < b$:

$$\frac{1}{b-a} \int_{a+s+c}^{b+s-c-1} \bar{f}_{[0,1]} \circ \phi^{-t}(x)dt \leq I_{a+s}^{b+s}(f)(x_-,x_+,w_x) \leq \frac{1}{b-a} \int_{a+s-c-1}^{b+s+c} \bar{f}_{[0,1]} \circ \phi^{-t}(x)dt$$
Taking $b$ to $\infty$ we see that as needed:

\begin{eqnarray*}
V(f)(u(x)) & = & \lim_{T\to \infty}\frac{1}{T}\int_a^{b+T} \bar{f}_{[0,1]}(\phi^{-t}x)dt \\
& = & \lim_{T \to \infty}I_{a+s}^{b+T+s}(f)(x_-,x_+,w_x) = \lim_{T \to \infty} I_{a}^{b+T}(f)(\xi,\eta,w)
\end{eqnarray*}
Since $f$ is non negative this also implies the finiteness of $I_a^b(f)(\xi,\eta,w)$ for any interval $[a,b]$ and $(\xi,\eta,w) \in A_f$. The case of general $f$ follows by linearity.

We now turn to the proof for the operator $J$. Denote $K_n = \{(\xi,\eta) \in S | (\xi,\eta) > n\}$. Let $f: S \to [0,\infty)$ be a $m\times\omega$ integrable function supported on bounded set. For large enough $n$ the support of $f$ is contained in $K_n$. Recall that by lemma \ref{bound on tau minus sigma} there exist constants $C_n$ such that if $(\xi,\eta), (g\xi,g\eta) \in K_n$ then:

$$|\tau(g,\xi,\eta) - \sigma(g,\eta)| \leq C_n$$
 It follows that for large $n$, any $(\xi,\eta,w) \in K_n\times\Omega$ and any $a<b$:
 
 $$\frac{b-a
 -2C_n}{b-a}I_{a+C_n}^{b-C_n}(f)(\xi,\eta,w) \leq J_a^b(f)(\xi,\eta,w)\leq \frac{b-a
 +2C_n}{b-a}I_{a-C_n}^{b+C_n}(f)(\xi,\eta,w)$$
 Letting $b$ tend to $\infty$ we get that for every $n$, $(\xi,\eta,w) \in A_f \cap K_n\times\Omega$ and $a<b$:
 
 $$\lim_{T \to \infty}J_a^{b+T}(f)(\xi,\eta,w) = V(f)(v(\xi,\eta,w))$$
Since $n$ is arbitrary and $K_n$ exhaust the space we see that the convergence actually holds on all of $A_f$. This also implies that $J_a^b(f)(\xi,\eta,w)$ is finite for any interval $[a,b]$. Finally the case of general $f$ follows by linearity.
\end{proof}

From this point onward it will be convenient to use the space $\partial^2 G$ and not $S$. Since $S$ is $G$-invariant and of full measure all the results proved thus far transfer immediately (although in the proof of the existence of $\hat{X}$ it was important that certain relations hold everywhere and not only almost everywhere, and now $\hat{X}$ will intersect only almost every $G$-orbit).

The last ingredient we are missing for the proof of ergodicity is the following contraction lemma from \cite[Lemma~2.6]{BF17}:

\begin{lemma}(\cite[Lemma~2.6]{BF17}) \label{contraction for Hopf argument}
For any compact $K \subset \partial^2 G$ there exists a constant $C_K$ such that for any $\xi,\xi',\eta \in \partial G$ and $g \in G$ satisfying :

$$(\xi,\eta), (\xi',\eta), (g\xi,g\eta) \in K, \hspace{20pt} \sigma(g,\eta) > 0$$
we have:

$$\sigma(g,\xi),\sigma(g,\xi') \in [-\sigma(g,\eta)-C_K,-\sigma(g,\eta)+C_K],  \hspace{20pt}  d_\epsilon(g\xi,g\xi') < e^{-\epsilon \sigma(g,\eta)+C_K}$$

\end{lemma}

\begin{proof}
Because $(\xi,\eta), (\xi',\eta), (g\xi,g\eta) \in K$ and $K$ is compact we have that $(\xi,\eta)$, $(\xi',\eta)$, $(g\xi,g\eta) \approx_K 0$.
Since $\sigma(g,\eta) > 0$ we see that $(g^{-1},\eta) > \frac{|g|}{2}$ but $\min\{(\xi,g^{-1}),(g^{-1},\eta)\} \lesssim (\xi,\eta) \approx_K 0$ so we conclude that $(g^{-1},\xi) \approx_K 0$. By lemma \ref{basic contraction lemma} we deduce:

$$0 \approx_K (g\xi,g\eta) \gtrsim \min\{(g\xi,g),(g,g\eta)\} \approx |g| - \max\{(g^{-1},\xi),(g^{-1},\eta)\}$$
Since $(g^{-1},\xi) \approx 0$ it follows that $(g^{-1},\eta) \approx_K |g|$. 
Now note that $\min \{(\xi,g^{-1}),(g^{-1},\eta)\} \lesssim (\xi,\eta) \approx_K 0$ so since $(g^{-1},\eta) \approx_K |g|$ it follows that $(\xi,g^{-1}) \approx_K 0$ and similarly $(\xi',g^{-1}) \approx_K 0$. Therefore:

$$\sigma(g,\xi) \approx_K \sigma(g,\xi') \approx_K -\sigma(g,\eta) \approx_K -|g|$$
Using  \ref{basic contraction lemma} again we see that $(g\xi,g)  \approx |g| - (g^{-1},\xi) \approx_K |g|$ and similarly $(g\xi',g) \approx_K |g|$ so finally:

$$(g\xi,g\xi') \gtrsim \min\{(g\xi,g),(g,g\xi')\} \approx_K |g| \approx_K \sigma(g,\eta)$$
as needed.

\end{proof}

Now we will finally prove ergodicity.

\begin{theorem} \label{weak mixing}
For any ergodic p.m.p action $G\acts (\Omega,\omega)$ the diagonal action on $(\partial^2 G \times \Omega, m\times \omega)$ is ergodic.
\end{theorem}

\begin{proof}
To show ergodicity we will show that $L^1(Y,\beta)$ is one dimensional. In order to do this we will show that for any $f \in L^1(\partial^2 \times \Omega, m\times \omega)$, $V(f)$ is constant. Since $V$ is a bounded operator it suffices to show this for $f$ from a dense subspace of $L^1(\partial^2 \times \Omega, m\times \omega)$. We will use the image of the space $C_c(\partial^2 G,m)\otimes L^1(\Omega,\omega)$ under the injection $\phi \otimes \psi \mapsto \phi \cdot \psi$. It is enough to show the image of pure tensors under $V$ is constant. Let $f = \phi \cdot \psi$ with $\phi \in C_c(\partial^2 G,m), \psi \in L^1(\Omega,\omega)$ and $\bar{f} = V(f) \circ v$. We will show that $\bar{f}$ is constant. Note that $\bar{f}$ need not be in $L^1(\partial^2G \times \Omega, m\times \omega)$. Our strategy will be to prove that $\bar{f}(\xi,\eta,w)$ is independent of $\xi$ and by symmetry also independent of $\eta$ so $\bar{f}$ is actually given by a $G$-invariant function of $w$ which is constant by ergodicity.

Fix a compact set $K \subset \partial^2 G$ with $supp(\phi) \subset int(K)$ and $\delta > 0$. Denote $h = \mathbbm{1}_K \cdot|\psi| \in L^1(\partial^2 \times \Omega, m\times \omega)$ and $\bar{h} = V(h)\circ v$. We will show that for $\mu\times\omega$-almost every $(\eta,w)$ there exists a $\mu$-full measure subset such that if $\xi, \xi'$ are in this subset then:

\begin{equation} \label{almost independant of xi}
  (\xi',\eta),(\xi,\eta)\in K \implies  |f(\xi,\eta,w) - f(\xi',\eta,w)| < \delta  \cdot \bar{h}(\xi,\eta,w)
\end{equation}
By theorem \ref{ergodic theorem for V} there exists a full measure subset $A$ of $\partial^2 G \times \Omega$ such that for all $(\xi,\eta,w) \in A$ and any $a \in \R$:
$$\lim_{T\to \infty}J_a^{a+T}(f)(\xi,\eta,w) = \bar{f}(\xi,\eta,w)$$

$$\lim_{T\to \infty}J_a^{a+T}(h)(\xi,\eta,w) = \bar{h}(\xi,\eta,w)$$
By Fubini's theorem it is enough to prove implication \ref{almost independant of xi} for $(\xi,\eta,w),(\xi',\eta,w) \in A$.

Fix $\xi,\xi',\eta,w$ such that $(\xi,\eta,w),(\xi',\eta,w) \in A$. By lemma \ref{contraction for Hopf argument} and using the uniform continuity of $\phi$ there exists $0 < a_1$ such that for all $g$ with $\sigma(g,\eta) > a_1$, if $(\xi,\eta),(\xi',\eta),(g\xi,g\eta) \in K$ then $|\phi(g\xi,g\eta) - \phi(g\xi',g\eta)| < \delta$, so:

$$|f(g\xi,g\eta,gw)-f(g\xi',g\eta,gw)| < \delta  \cdot h(g\xi,g\eta,gw)$$
On the other hand it follows from lemma \ref{contraction for Hopf argument} (after replacing the roles of $\xi$ and $\xi'$) and the fact that $supp(\phi) \subset int(K)$, that there exists $0 < a_2$ such that for all $g$ with $\sigma(g,\eta) > a_2$ if $(\xi,\eta),(\xi',\eta) \in K$ but $(g\xi,g\eta) \notin K $ then $(g\xi,g\eta)(g\xi',g\eta) \notin supp(\phi)$. Therefore:

$$|f(g\xi,g\eta,gw)-f(g\xi',g\eta,gw)| = \delta \cdot h(g\xi,g\eta,gw) = 0$$
So if $a = max\{a_1,a_2\}$ then for any $g$ with $\sigma(g,\eta) > a$, if $(\xi,\eta),(\xi',\eta)\in K$ then:

$$|f(g\xi,g\eta,gw)-f(g\xi',g\eta,gw)| < \delta \cdot h(g\xi,g\eta,gw)$$
Now if $(\xi,\eta),(\xi',\eta)\in K$ then:

\begin{align*}
    & |\bar{f}(\xi,\eta,w)-\bar{f}(\xi',\eta,w)| 
= |\lim_{T\to \infty}J_a^{a+T}(f)(\xi,\eta,w) - \lim_{T\to \infty}J_a^{a+T}(f)(\xi',\eta,w)|
\\\leq & \lim_{T\to\infty}\frac{1}{T} \int_{\{g\in G| \sigma(g,\eta) \in [a,a+T]\}} |f(g\xi,g\eta,gw)-f(g\xi',g\eta,gw)|d\lambda(g)
\\\leq & \lim_{T\to\infty}\frac{1}{T} \int_{\{g\in G| \sigma(g,\eta) \in [a,a+T]\}} \delta \cdot h(g\xi,g\eta,gw)d\lambda(g)
\\= & \delta \cdot \lim_{T\to\infty} J_a^{a+T}(h)(\xi,\eta,w) = \delta\cdot\bar{h}(\xi,\eta,w)
\end{align*}
Showing implication \ref{almost independant of xi} as needed.

Since $K$ and $\delta$ are general we can take a sequence $K_n$ exhausting $\partial^2 G$ and $\delta_n \to 0$. Applying implication \ref{almost independant of xi} to these sequences we conclude that for $\mu\times\omega$ almost every $(\eta,w)$ there exists a $\mu$ full measure subset such that if $\xi, \xi'$ are in this subset then: 

$$\bar{f}(\xi,\eta,w)=\bar{f}(\xi',\eta,w)$$
By symmetry the same statement holds reversing the roles of $\xi$ and $\eta$.

Using Fubini's theorem we conclude that for $\omega$-almost every $w$ there exist $\mu$-full measure subsets $A,B \subseteq \partial G$ such that for $\xi \in A$, $\bar{f}(\xi,\cdot,w)$ is an essentially constant function of $\eta$ and for $\eta \in B$, $\bar{f}(\cdot,\eta,w)$ is an essentially constant function of $\xi$. Therefore for $\omega$-almost every $w$, $\bar{f}(\cdot,\cdot,w)$ is essentially constant and thus $\bar{f}$ depends only on the $\Omega$ coordinate.
Since $\bar{f}$ is $G$-invariant and $G$ acts ergodicaly on $(\Omega,\omega)$ we conclude $\bar{f}$ is essentially constant as needed.

\end{proof}

Having proved ergodicity we know that the space $Y$ is a single point and the operator $V$ is simply given by $V(f) = \int f(\xi,\eta,w) dm\times\omega$. We can therefore re-state theorem \ref{ergodic theorem for V} as follows:

\begin{theorem}
    For any $f \in L^1(\partial^2 G \times \Omega,m \times \omega)$, for almost every $(\xi,\eta,w)$ the values $I_a^b(f)(\xi,\eta,w)$ and $J_a^b(f)(\xi,\eta,w)$ are finite for any interval $[a,b]$ and:
    $$\lim_{T \to \infty} I_a^{b+T}(f)(\xi,\eta,w) = \int f(\xi,\eta,w) dm\times\omega$$
    If in addition $f$ is supported on a bounded set then:

    $$\lim_{T \to \infty} J_a^{b+T}(f)(\xi,\eta,w) = \int f(\xi,\eta,w) dm\times\omega$$
\end{theorem}

\appendix

\section{Cocycles and Almost Cocycles} \label{appendix a}

Let $G$ be a locally compact second countable group with a measure class preserving action on a standard Lebesgue space $(X,\mu)$.
A Borel \textbf{cocycle} of the action $G \acts X$ with values in $\R$ is a Borel function $\alpha:G \times X \to \R$ such that for every $g,h \in G$, for almost every $x \in X$:

$$\alpha(gh,x) = \alpha(g,hx) + \alpha(h,x)$$
A cocycle is called strict if the above equation holds for every $x$.
Two cocycles $\alpha, \rho$ are called cohomologous if there exists a function $F:X \to \R$ such that for every $g$, for almost every $x$:

$$\alpha(g,x) - \rho(g,x) = F(gx) - F(x)$$
If $\alpha,\rho$ are strict cocycles they are strictly \textbf{cohomologous} if the above equation holds for every $x$.

A Borel \textbf{almost cocycle} of the action $G \acts X$ with values in $\R$ is a Borel function $\sigma:G \times X \to \R$ such that for every $g,h \in G$, for almost every $x \in X$:

$$\sigma(gh,x) \approx \sigma(g,hx) + \sigma(h,x)$$
An almost cocycle is called strict if the above equation holds for every $x$.

\begin{lemma} \label{almost cocycle lemma}
Let $G$ be a second countable locally compact group with a measure class preserving action on a standard Lebesgue space $(X,\mu)$. Suppose that $\alpha,\sigma: G \times X : \to \R$ are a strict Borel cocycle and a strict Borel almost cocycle respectively such that for every $g \in G$, for almost every $x \in X$, $\alpha(g,x) \approx \sigma(g,x)$. There exists a strict cocycle $\rho$ strictly cohomologous to $\alpha$ and a $G$-invariant full measure Borel set $S \subseteq X$ such that for all $(g,x) \in G \times S$, $\rho(g,x) \approx \sigma(g,x)$.
\end{lemma}

\begin{proof}
For $x \in X$ denote $E_x(C) = \{g\in G ||\alpha(g,x) - \sigma(g,x)| < C\}$. For every $g \in G$, for almost every $x \in X$, $\alpha(g,x) \approx \sigma(g,x)$ so by Fubini's theorem there exists a full measure Borel set $A \subseteq X$ and  $C_1 >0$ such that for every $x \in A$, $E_x(C_1)$ has full Haar measure. By \cite[Appendix~B.8]{ZIM84} there exists a full measure Borel subset $B \subseteq A$ such that $S = G.B$ is Borel. 
If $E_x(C_1)$ has full Haar measure then there exists $C_2 > 0$ such that for almost every $g \in G$, $E_{gx}(C_2)$ has full Haar measure. Indeed if $g \in E_x(C_1)$ then for any $h \in E_x(C_1)g^{-1}$:
$$\alpha(h,gx) = \alpha(hg,x) - \alpha(g,x) \approx \sigma(hg,x) -\sigma(g,x) \approx \sigma(h,gx)$$
Since every $x \in S$ has an orbit which intersects $B$ we conclude that for any $x \in S$, for almost all $g \in G$, $E_{gx}(C_2)$ has full Haar measure.

We claim that there exists a Borel function $F: S \to \R$ and $C_3 >0$ such that for every $x \in S$, for almost every $g \in G$, $|\alpha(g,x) - \sigma(g,x) - F(x)| < C_3$. Note that for almost every $x \in B$ we can take $C_3 = C_1$ and $F(x) = 0$ but the point is that the formula holds over $S$ which is $G$ invariant. To see this notice that if $x \in S$ then for almost every $g \in G$, $E_{gx}(C_2)$ has full measure, so for almost every $g$, for almost every $h$, $|\alpha(h,gx) - \sigma(h,gx)| < C_2$. For such $g,h$:
$$\alpha(hg,x)-\alpha(g,x) = \alpha(h,gx) \approx \sigma(h,gx) \approx \sigma(hg,x)-\sigma(g,x)$$
Taking $k = hg$, re-arranging terms and using Fubini's theorem we get that in particular for almost every $(g,k) \in G^2$:
$$\alpha(g,x) - \sigma(g,x) \approx \alpha(k,x) - \sigma(k,x)$$
Since each side depends only on one of the coordinated $(g,k)$ we conclude that there exists a full measure subset $D_x \subseteq G$ such that for all $g,k \in D_x$:

$$\alpha(g,x) - \sigma(g,x) \approx \alpha(k,x) - \sigma(k,x)$$

Choosing any probability measure $\theta$ on $G$ in the class of Haar measure we can define $F(x) = \int \alpha(g,x) - \sigma(g,x) d\theta(g)$ (we only do this to ensure $F$ is Borel). 
Since for every $x$, for almost every $g \in G$, $E_{gx}(C_2)$ has full measure, we conclude that for every $x\in S$, $|F(gx)| < C_2$ for almost every $g \in G$. For convenience we extend $F$ to $X$ by putting $F(x) = 0$ for $x \notin S$.

We can now define the cocycle $\rho$ by:
$$\rho(g,x) = \alpha(g,x) - F(x) + F(gx)$$
By definition $\alpha$ and $\rho$ are 
strictly cohomologous. Since for every $x \in S$, for almost every $g \in G$, $|F(gx)| < C_2$. We conclude by taking $C_4 = C_2 + C_3$, that for all $x \in S$, for almost every $g \in G$, $|\rho(g,x) - \sigma(g,x)| < C_4$.

We now show that for all $(g,x) \in G\times S$, $\rho(g,x) \approx \sigma(g,x)$ uniformly in $g,x$. Indeed given $(g,x) \in G \times S$, there exists a full measure set of $h \in G$ such that $\rho(h,gx) \approx \sigma(h,gx)$ and $\rho(hg,x) \approx \sigma(hg,x)$.  (Here we are using the critical fact that $gx \in S$ which would not be true for $A$ or $B$). For such an $h$:
$$\rho(g,x) = \rho(hg,x) - \rho(h,gx) \approx \sigma(hg,x) - \sigma(h,gx) \approx \sigma(g,x)$$
So for every $(g,x) \in G \times S$, $\rho(g,x) \approx \sigma(g,x)$, as needed.

\end{proof}

\section{Actions Admitting a Borel Cross Section} \label{appendix b}

Let $G$ be a locally compact second countable group acting on a standard measure space with sigma finite measure $(Z,m)$. Assume that the stabilizer of almost every point in $Z$ is compact and that the action has a cross section, i.e. a measurable subset $\hat{X} \subset Z$ intersecting every orbit exactly once. Denote $X = Z//G$ the space of $G$-ergodic components and denote by $p$ the projection $p:Z \to X$. Since bijective Borel maps are Borel isomorphisms, $p$ restricts to a Borel isomorphism between $\hat{X}$ and $X$ so $\hat{X}$ is endowed with a canonical measure class and a measurable projection $Z \to \hat{X}$ sending every point to the point in $\hat{X}$ sharing the same orbit. In particular every fiber over $X$ is a $G$-orbit. Finally assume that there exists a compact identity neighborhood $V \subset G$ such that the measure $p_*(m|_{V\hat{X}})$ is $\sigma$-finite on $X$.

Fix a left Haar measure $\lambda$ on $G$. Let $B$ be a transitive $G$-space with compact stabilizers. Recall that up to a positive scalar there is a unique $\sigma$-finite $G$-invariant measure on $B$. Choosing a base point $a \in A$ we can construct this measure as the image of $\lambda$ under the orbit map $g \mapsto gb$. Choosing a different base point can only change the measure by a scalar and if $G$ is unimodular it dose not depend on the base point at all. To sum up, fixing a left Haar measure gives a canonical choice of invariant measure on pointed transitive $G$-spaces with compact stabilizers. This measure is independent of the base point when $G$ is unimodular.

Since $\hat{X}$ gives a choice of base point from every $G$ orbit in $Z$, for almost every $x \in X$ we have a canonical $G$-invariant measure $m_x$ supported on the fiber over x. If $G$ is unimodular the measures $m_x$ do not depend on $\hat{X}$. Denote by $K_x$ the stabilizer of the point in $\hat{X}$ in the fiber over $X$. An explicit formula for $m_x$ is $m_x(A\hat{X}) = \lambda(A K_x)$ for any measurable $A \subset G$. 

\begin{theorem} \label{measure construction}
Given $G$, $Z$ and $\hat{X}$ as above, there exists a unique measure $\nu$ on $Y$  such that $m = \int \! m_x d\nu(y)$. The measure $\nu$ is in the measure class of $p_*m$ and if G is unimodular it dose not depend on $\hat{X}$.
\end{theorem}

\begin{proof}
Denote by $\tau$ the measure $p_*(m|_{V\hat{X}})$. Since $G$ is second countable we can write $G = \bigcup_{i\in \N} g_iV$ so $\tau$ and $p_*m$ are in the same measure class (but $p_*m$ might not be $\sigma$-finite). Since $\tau$ and $m$ are $\sigma$-finite there exists by \cite[Theorem~6.3]{DS12} a unique disintegration $m = \int\!\alpha_xd\tau(x)$. Now for every $g \in G$, $\int\!\alpha_xd\tau(x) = m = g_*m = \int\!g_*\alpha_xd\tau(x)$. By uniqueness of the disintegration, almost every $\alpha_x$ is invariant. Since almost every  $\alpha_x$ is supported on the fiber $p^{-1}(x)$, which is a $G$-orbit, there exists $c(x) > 0$ such that $\alpha_x = c(x)m_x$. $c(x)$ is a measurable function since $c(x) = \frac{\alpha_x(V\hat{X})}{m_x(V\hat{X})} = \frac{\alpha_x(V\hat{X})}{\lambda(VK_x)}$, $\alpha_x$ is a measurable family of measures and $\lambda(VK_x)$ is a measurable function. Therefore $m_x$ is a measurable family of measures.
For any $f\in L^1(m)$ we have $\int\!f(z)dm(z) = \int\!\int\!f(z)c(x)dm_x(z)d\tau(x)$ so defining $d\nu = c(x)d\tau$ we get $m = \int\!m_xd\nu(x)$. Since $c(x) > 0$ for almost every $x$, $\nu$ and $\tau$ are in the same measure class which is the measure class of $p_*m$. $\nu$ is unique since if $\nu'$ is a measure satisfying $m = \int \! m_x d\nu'(x)$ then for any measurable $f \in L^1(\nu')$, $\int\!f(x)d\nu'(x) = \int\! \int_{V\hat{X}}\!\frac{f(p(z))}{m_{p(z)}(V\hat{X})}dm_x(z) d\nu'(x) = \int\!\frac{f(p(z))}{m_{p(z)}(V\hat{X})}dm(z)$ depends only on $m$ and the $m_x$ which only depend on $\hat{X}$.
Finally If G is unimodular the measures $m_x$ do not depend on $\hat{X}$ so since there is a unique measure $\nu$ such that $m = \int \! m_x d\nu(x)$, $\nu$ is independent of $\hat{X}$.
\end{proof}

\begin{theorem}  \label{invariant measure on quotient}
 Suppose in the setting above that the space $(Z,m)$ has a measure preserving action of a group $H$ commuting with the $G$ action. If $G$ is unimodular the natural $H$ action on $Y$ preserves the measure $\nu$.
\end{theorem}

\begin{proof}
$H$ acts by automorphisms on the measure preserving system $(Z,m,G)$, thus for any $h \in H$, $h_*\nu$ is the canonical measure on $X$ associated to the cross section $h\hat{X}$ by theorem \ref{measure construction}. Since $G$ is unimodular the measure is independent of the cross section so this is the same as the measure associated to the cross section $\hat{X}$, i.e. $\nu$. So $h_*\nu = \nu$ as needed.
\end{proof}

\bibliographystyle{amsplain}
\bibliography{boundary_reps_bibliography}
\end{document}